\documentclass[12pt,reqno,oneside]{amsart}

\usepackage{amsmath,amsfonts,amssymb,amsthm}
\usepackage{mathrsfs}
\usepackage{graphics,graphicx}
\usepackage[usenames,dvipsnames]{xcolor}
\usepackage{times}
\usepackage{bbm}
\usepackage[rmargin=1.2in,lmargin=1in,tmargin=1in,
	bmargin=1in,marginparwidth=1.17in,marginparsep=0.02in]{geometry}
\usepackage{enumerate,enumitem}
\usepackage{url}
\usepackage{dsfont}
\usepackage{microtype}
\usepackage{mhequ}
\usepackage{cprotect}
\usepackage{margincomment}
\usepackage{xstring}

\usepackage[colorlinks=true, pdfstartview=FitV, linkcolor=BrickRed,citecolor=black, urlcolor=black]{hyperref}
\usepackage{orcidlink}

\newtheorem{theorem}{Theorem}
\newtheorem{proposition}{Proposition}[section]
\newtheorem{lemma}[proposition]{Lemma}
\newtheorem{corollary}[proposition]{Corollary}
\theoremstyle{definition}
\newtheorem{definition}[proposition]{Definition}

\newtheorem{remark}[proposition]{Remark}
\numberwithin{equation}{section}

\colorlet{darkblue}{blue!90!black}
\colorlet{darkgreen}{green!50!black}

\def\restr{\mathbin{\upharpoonright}}

\newcommand\eps{\varepsilon}
\newcommand\e{{\rm e}}
\newcommand\dd{{\rm d}}
\newcommand\ddt{{\frac{\dd}{\dd t}}}
\def\Re{\mathop{\mathrm{Re}}}

\def\cl{\mathop{\mathrm{cl}}}
\def\Id{\mathord{\mathrm{Id}}}

\DeclareMathOperator{\Ker}{Ker}
\DeclareMathOperator{\tr}{tr}
\DeclareMathOperator{\supp}{supp}
\DeclareMathOperator{\Span}{span}

\def\l {\langle}
\def\r {\rangle}
\newcommand\de{{\partial}}
\newcommand\EE{{\mathbb E}}
\newcommand\PP{{\mathbb P}}

\newcommand\sym {{H_{\mathrm{sym}}}}
\newcommand\inv{\mathrm{inv}}
\newcommand\pp{\mathrm{pp}}
\newcommand\cont{\mathrm{cont}}
\newcommand\loc{\mathrm{loc}}

\newcommand{\NN}{\mathbb{N}}
\newcommand{\ZZ}{\mathbb{Z}}
\newcommand\TT {{\mathbb T}}
\newcommand\RR {{\mathbb R}}

\newcommand\cC{{\mathcal C}}
\newcommand\cD{{\mathcal D}}
\newcommand\cF{{\mathcal F}}
\newcommand\cK{{\mathcal K}}
\newcommand\cO{{\mathcal O}}
\newcommand\cS{{\mathcal S}}

\newcommand{\eqdef}{\stackrel{\mbox{\tiny\rm def}}{=}}

\makeatletter

\DeclareRobustCommand{\TitleEquation}[2]{\texorpdfstring{\StrLeft{\f@series}{1}[\@firstchar]$\if%
b\@firstchar\boldsymbol{#1}\else#1\fi$}{#2}}

\renewcommand\subsubsection{\@startsection{subsubsection}{3}%
\normalparindent{.5\linespacing\@plus.7\linespacing}{-.5em}
{\normalfont\bfseries}}

\def\@tocline#1#2#3#4#5#6#7{\relax
  \ifnum #1>\c@tocdepth 
  \else
    \par \addpenalty\@secpenalty\addvspace{#2}%
    \begingroup \hyphenpenalty\@M
    \@ifempty{#4}{%
      \@tempdima\csname r@tocindent\number#1\endcsname\relax
    }{%
      \@tempdima#4\relax
    }%
    \parindent\z@ \leftskip#3\relax \advance\leftskip\@tempdima\relax
    \rightskip\@pnumwidth plus4em \parfillskip-\@pnumwidth
    #5\leavevmode\hskip-\@tempdima
      \ifcase #1
       \or\or \hskip 1em \or \hskip 2em \else \hskip 3em \fi%
      #6\nobreak\relax
    \dotfill\hbox to\@pnumwidth{\@tocpagenum{#7}}\par
    \nobreak
    \endgroup
  \fi}
\makeatother

\title[Stochastic RAGE and Enhanced Dissipation]{A Stochastic RAGE Theorem and Enhanced Dissipation \\ for Transport Noise}

\author[M. Coti Zelati]{Michele Coti Zelati\orcidlink{0000-0002-2495-2212}}
\address{Department of Mathematics, Imperial College London}
\email{m.coti-zelati@imperial.ac.uk}

\author[M. Hairer]{Martin Hairer\orcidlink{0000-0002-2141-6561}}
\address{EPFL, Lausanne, Switzerland and Imperial College London, UK}
\email{martin.hairer@epfl.ch}

\author[D. Villringer]{David Villringer\orcidlink{0009-0002-1681-4992}}
\address{Department of Mathematics, Imperial College London}
\email{d.villringer22@imperial.ac.uk}

\begin{document}



\begin{abstract}
We prove a stochastic version of the classical RAGE theorem that applies to the two-point motion generated by noisy transport equations. As a consequence, we identify a necessary and sufficient condition for the corresponding diffusive equation to be dissipation enhancing. This involves the identification of a non-trivial, finite dimensional subspace that is invariant for the family of self-adjoint operator characterizing the structure of the transport noise. We discuss several examples and prove a sharp enhanced dissipation rate for stochastic shear flows.
\end{abstract}

\maketitle
\thispagestyle{empty}

\setcounter{tocdepth}{1}
\tableofcontents

\section{Diffusion and Mixing in Fluid Flows}
The evolution of a \emph{passive scalar} $f:[0,\infty)\times \TT^d\to \RR$ subject to drift in a given  incompressible velocity field $u:[0,\infty)\times \TT^d\to \RR^d$ and diffusion is described by the advection-diffusion equation
\begin{equation}\label{eq:ADE}
    \de_t f+u\cdot \nabla f=\nu \Delta f\:.
\end{equation}
This model, widely used to describe fluctuations around equilibrium of physically relevant quantities such as fluid density or temperature, retains two fundamental stabilization mechanisms that drive the system to its average: inviscid mixing, due to the presence of transport, and enhanced diffusion, due to the interaction of transport and diffusion of strength $\nu>0$. 

Given an initial datum $f_0:\TT^d\to \RR$ and under the evolution-preserved mean-free condition
\begin{equation}\label{eq:meanfree}
    \int_{\TT^d} f_0(x)\dd x=0\:,
\end{equation}
the corresponding energy equation
\begin{equation}\label{eq:diffener}
    \frac12\ddt \|f\|_{L^2}^2+\nu \|\nabla f\|_{L^2}^2=0
\end{equation}
implies that the solution vanishes as $t\to\infty$ exponentially fast, at a rate at least proportional to $\nu$. In other words, the characteristic diffusive time-scale $\cO(\nu^{-1})$ appears readily from combining Poincaré's inequality with Grönwall's lemma  in \eqref{eq:diffener}.

As \eqref{eq:diffener} is in particular true also when $u=0$, a crucial missing information in the above reasoning is the influence of a (possibly chaotic) velocity field $u$ on dissipation time scales. In fact, the appearance of shorter \emph{enhanced dissipation} 
time scales $\cO(T_\nu)$, with $T_\nu\ll \nu^{-1}$, has fascinated the scientific 
community since the late nineteenth century, although  mathematicians 
have only recently started looking at these questions from a rigorous perspective. Without any claim 
to exhaustiveness, we mention the works \cite{Constantin_Kiselev_Ryzhik_Zlatos_2008,BW13,CZDE20,Wei21,Seis23}, and we postpone the discussion of specific references to the next sections.

The aim of this work is to extend the seminal result of \cite{Constantin_Kiselev_Ryzhik_Zlatos_2008} on the characterization of enhanced dissipation for autonomous velocity fields, incorporating the case of random drifts driven by Gaussian processes, commonly referred to as \emph{transport noise}.



\subsection{Transport Noise and Enhanced Dissipation}\label{sub:transportnoise}
In this article, we consider \eqref{eq:ADE} on $\TT^d$ in the specific case of transport noise and we write 
\begin{equation}\label{eq:StochAdvDiff}
\dd f^\nu+\sum_{k=1}^{\infty} \sigma_k\cdot \nabla f^\nu \circ \dd W^k_t=\nu \Delta f^\nu \dd t\;,
\end{equation}
where $\circ$ denotes Stratonovich  integration. Here, $\{W^k_t\}_{k\in \NN}$ is a collection of independent standard Brownian motions and $\Sigma:=\{\sigma_k\}_{k\in \NN}$ is a collection of $C^2$ divergence-free vector fields. We solve \eqref{eq:StochAdvDiff} for initial conditions $f_0\in H\subset L^2(\TT^d)$, where $H$ is a closed subspace invariant under the evolution: 
in general, it consists of the mean-free functions, but it could further restrict the class of initial data, 
as in the case of shear flows (see Theorem~\ref{thm:shearEn} below).
Note that as a consequence of the smoothing property of \eqref{eq:StochAdvDiff}, $H$ is necessarily such
that $H \cap C^\infty$ is dense in $H$.

The definition of enhanced dissipation, adapted to the transport noise case, can be stated as follows: $\Sigma$ is dissipation enhancing if for any $\tau,\delta>0$, there exists $\nu_0>0$ such that for every $\nu\in(0,\nu_0)$ we have 
\begin{equation}\label{eq:decayEL2}
    \EE\| f^\nu(\tau/\nu)\|^2_{L^2}\leq \delta \| f_0\|^2_{L^2}\:,
\end{equation}
for every initial datum $f_0\in H$. Our main result provides a necessary and sufficient condition for enhanced dissipation in the stochastic setting.


%

\begin{theorem}\label{main result}
Exactly one of the following two statements holds true:
\begin{itemize}
\item The drift defined by $\Sigma:=\{\sigma_k\}_{k\in \NN}$ is dissipation enhancing.
\item There exists a non-trivial finite dimensional subspace \(V\subset H\cap H^1(\TT^d)\) so that \((\sigma_k \cdot \nabla) V \subset V\), for all \(k\in \NN\).
\end{itemize} 
\end{theorem}

\begin{remark}
Despite the fact that our condition for enhanced dissipation seems rather abstract on the surface, it is actually checkable in certain 
situations of interest. Indeed, any such finite dimensional $V\subset H\cap H^1(\mathbb{T}^d)$ must be the finite span of 
eigenfunctions of the $\sigma_k \cdot \nabla$. This can easily be seen by noting that $\sigma_k \cdot \nabla $ restricted to $V$ is an 
anti-self adjoint matrix, and therefore admits a basis of eigenvectors on $V$. Therefore, checking the condition in 
Theorem~\ref{main result} essentially boils down to computing the $H^1(\mathbb{T}^d)$ eigenvectors of the $\sigma_k \cdot \nabla $, 
and seeing if any of their finite spans overlap.
\end{remark}


Theorem~\ref{main result} is a consequence of a more general result, which gives a necessary and sufficient condition for enhanced dissipation in an abstract setting for a general SPDE, of which \eqref{eq:StochAdvDiff} is a particular instance. Fix a 
Hilbert space \(H\) and consider a collection of anti self-adjoint operators \(\{iL_k\}_{k\in \NN}\) and a positive, unbounded operator \(A\) with discrete spectrum tending to infinity. 
Obtaining a criterion for enhanced dissipation for abstract SPDEs
of the form
\begin{equation}\label{eq:abstrdiff}
\dd f^\nu+\sum_{k=1}^\infty i L_k f^\nu\circ \dd W^k_t+\nu Af^\nu \dd t=0
\end{equation}
necessitates the study of the corresponding \textit{inviscid} problem, given by 
\begin{equation}\label{eq:InviscidAbstract}
\dd f+\sum_{k=1}^\infty  iL_kf \circ \dd W^k_t=0\:.
\end{equation}
In the particular case of transport noise, as in \eqref{eq:StochAdvDiff} for $\nu=0$, equations of this form are intrinsically amenable to analysis, since one may associate to them a stochastic flow of characteristics (see \cite{Kunita_1997}), given by 
\begin{equation}
\phi_{s,t}(x)=x+\sum_{k=1}^\infty\int_{s}^t\sigma_k(\phi_{s,z})\circ \dd W^k_z\:,
\end{equation}
and then solutions are simply given by composition of $f_0$ with the inverse flow map $\phi_{0,t}^{-1}$. 

In the general case we assume that for every $\nu \ge 0$ there exists a bounded random solution map \(\Phi^{\nu}_{s,t}:H \to H\) so that for any \(f_0 \in L^2(\Omega, \cF_s,\PP,H)\), the process \(t \mapsto\Phi^{\nu}_{s,t}f_0\) is a weak solution to \eqref{eq:abstrdiff} with initial data \(f_0\) assigned at time $s\in \RR$. In the case $\nu =0$ we drop the superscript 
and simply write $\Phi_{s,t}$. Under mild regularity assumptions on this map, we can then obtain the following stochastic version of the classical RAGE theorem (see \cite{Reed_Simon_1972}).
\begin{theorem}\label{thm:sRAGE}
Let \(H_\inv\) be the smallest closed subspace of \(H\) containing all finite dimensional subspaces that are invariant under the actions of the \(iL_k\), and denote by \(P_c\) the orthogonal projection onto \(H_\inv^\perp\). Then, \(P_c\) commutes with \(\Phi_{s,t}\) and, for any $f_0 \in H$ and any compact operator \(K: H \to H\), it holds that
\begin{equation}
\lim_{t \to \infty}\EE\|K P_c \Phi_{s,t}f_0\|^2 = 0\:,  
\end{equation} 
for every $s\in \RR_+$.
\end{theorem}
\begin{remark}
Theorem~\ref{thm:sRAGE} is sharp, in the sense that if \(f_0\) is such that \((I-P_c)f_0 \neq 0\), then there exists a compact operator \(K:H \to H\) so that \(\displaystyle\liminf_{ t \to \infty} \EE\|K\Phi_{s,t}f_0\|^2 >0\).
\end{remark}
\begin{remark}
In the case \(L_k=0\) for every $k\geq 2$, the space \(H_{\inv}\) corresponds precisely to the space \(H_{\pp}\) of vectors for which the corresponding spectral measure is pure point, and so \(P_c\) is the orthogonal projection onto the continuous spectral subspace. In this simple case, the flow map admits the  representation \(\Phi_{s,t}=\e^{iL(W_t-W_s)}\), and so one can argue in much the same way as in the proof of the classical RAGE theorem, obtaining Theorem \ref{thm:sRAGE} by Wiener's theorem for measures (see e.g. \cite[Theorem XI.114]{Reed_Simon_1972}).
\end{remark}
\begin{remark}
Theorem~\ref{thm:sRAGE} allows us to deduce without much effort further properties about the properties of the stochastic transport equation. In particular, we deduce in Corollary \ref{invariant measures} that it admits nontrivial invariant measures on $H$ if and only if there do not exist any finite dimensional invariant subspaces of the $iL_k$.
\end{remark}

\subsection{Stochastic Shear Flows}
The general criterion established in Theorem~\ref{main result} is qualitative in nature and does not in general provide an explicit decay rate for the $L^2$ norm of solutions. As in the deterministic setting, shear flows provide an example of drifts in which explicit rates are obtainable. In close analogy with \cite{Bedrossian_CotiZelati_2017}, let $\{u_j(y)\}_{j\in \NN}$ be a family of smooth functions on $\TT$ and consider the SPDE on $\TT^2$ given by
\begin{equation}
\label{Eq:StochasticShear}
\dd f^\nu+\sqrt{2\kappa}\sum_{j=1}^\infty u_j(y)\partial_x f^\nu \circ \dd W^j_t=\nu \partial^2_y f^\nu \dd t, \qquad \kappa,\nu>0\:,
\end{equation}
with assigned initial condition \(f_0 \in L^2\). Denote by $n_0\in \NN_{\geq 0}$ the maximum order of \emph{spatially overlapping} critical points of the $u_j$, i.e.\ the smallest $n_0$ so that, for all $y \in \mathbb{T}$ there exists $n \leq n_0$ and a $j \in \NN$ so that $u_j^{n+1}(y) \neq 0$. Under some mild regularity assumptions, we then obtain the following result.

\begin{theorem}\label{thm:shearEn}
Let \(P_\ell\) be the projection onto the \(\ell^{\text{th}}\) Fourier mode in \(x\). Then, there exist deterministic constants $\eps_0, \alpha_0$, so that for any $\eps<\eps_0$, \(\sqrt{\frac{\kappa}{\nu}}|\ell|\geq \alpha_0\), there exists a random constant $C_{\eps,\ell,\nu,\kappa}$ so that
\begin{equation}\label{eq:enhancedShear}
\|P_\ell f^\nu(t)\|^2_{L^2_y} \leq C_{\eps,\ell,\nu,\kappa} \nu^{-\gamma}\kappa^\frac{\beta}{2}|\ell|^\beta\exp\Bigl({-\eps \nu^\frac{n_0+1}{n_0+2}\kappa^\frac{1}{n_0+2} |\ell|^\frac{2}{n_0+2} t}\Bigr)\|P_\ell f_0\|^2_{L^2_y}\:, \qquad \forall t\geq 0\:,
\end{equation}
for some universal, explicitly computable positive constants $\gamma, \beta >0$ depending only on $n_0$. 
For any $p<1+\eps^{-1}$, there exists a constant $c_{\eps,p}$ so that 
\begin{equation}
(\mathbb{E}(C_{\eps,\ell,\nu,\kappa})^p)^\frac{1}{p}\leq c_{p,\eps}\:,
\end{equation}
uniformly over $\ell,\nu,\kappa$.
Furthermore, assuming that $\sum_{j}\|u_j\|_{C^{2n_0+2}}<\infty$, there exists a sequence of initial data \(f_0^\alpha\), indexed by \(\alpha=\sqrt{\frac{\kappa}{\nu}}\ell\) so that the solution to \eqref{Eq:StochasticShear} with initial data \(f_0^\alpha\) satisfies 
the lower bound
\begin{equation}
\EE\|P_\ell f^\nu(t)\|_{L^2} \geq  \exp\Bigl({-\Tilde{c}\nu^\frac{n_0+1}{n_0+2}\kappa^\frac{1}{n_0+2}|\ell|^\frac{2}{n_0+2}t}\Bigr) \|P_\ell f_0^\alpha\|_{L^2}\:,
\end{equation}
for some constant $\Tilde{c}$ independent of $\alpha, \nu, \ell$.
\end{theorem}
\begin{remark}
The second part of Theorem~\ref{thm:shearEn} shows that the rates we obtain are sharp, up to possibly the factors in front of the exponential in \eqref{eq:enhancedShear}.
\end{remark}
\begin{remark}\label{rem:hypo}
Due to the fact that the above result only depends on the size of \(\alpha=\sqrt{\frac{\kappa}{\nu}}|\ell|\), this is quite meaningful even in the case of relatively large \(\nu\). Indeed, taking into account the parabolic regularisation in $y$, our result gives Gevrey-$p$ regularity of any parameter \(p >\frac{n_0+2}{2}\). Whilst some hypoelliptic regularisation is expected since the equation satisfies a stochastic analogue of H\"ormander's bracket condition (see e.g.\ \cite{Kunita_1994}), this gives us an exact quantification on the extent of this regularisation.
\end{remark}
\begin{remark}
The analogous question of determining rates of enhanced dissipation for deterministic shear flows of the form 
\begin{equation}\label{eq:detshear}
\partial_t f+u(y)\partial_x f =\nu \Delta f
\end{equation}
was tackled in \cite{Bedrossian_CotiZelati_2017}, and later in \cite{ABN22,CZG23,Villringer24,GLM24}. 
Whilst existing proofs use hypocoercivity \cite{Bedrossian_CotiZelati_2017}, a quantitative H\"ormander-type approach \cite{ABN22}, pseudo-spectral estimates \cite{CZG23}, Malliavin calculus \cite{Villringer24} or Lagrangian particles \cite{GLM24}, our methods are quite different. 
For \eqref{eq:detshear}, the determining factor for the rates at which dissipation is enhanced is the maximal order $n_0\in \NN$ of vanishing of $u'$ at critical points, and the exponential decay rate is proportional to \(\nu^\frac{n_0+1}{n_0+3}|\ell|^\frac{2}{n_0+3}\). Note that in the stochastic case, if there exists only a single $u_j$ which has derivative that vanishes up to order $n_0$, our results yield rates proportional to $\nu^\frac{n_0+1}{n_0+2}|\ell|^\frac{2}{n_0+2}$. In this sense, the enhanced dissipation for stochastic shear flows seems to be slower than in the deterministic case, whereas the hypoelliptic regularisation effect is actually stronger. Note however that in contrast to the deterministic case, $n_0$ is permitted to take the value $0$, (and it turns out this is in some sense the generic case).
\end{remark}

\subsection{Outline of the Article and Main Ideas}\label{sub:mainideas}
Most of the work to prove Theorem~\ref{main result} will lie in the consideration of the inviscid problem and the corresponding RAGE theorem, Theorem~\ref{thm:sRAGE}. 
The core observation is that one can obtain information on expressions of the form 
\begin{equation}
\EE|\langle \psi,\Phi_{s,t}f_0\rangle|^2\:, \qquad \psi, f_0\in H\:,
\end{equation}
by re-writing the above expression in terms of the corresponding \emph{two-point motion} 
\begin{equation}
\EE|\langle \psi,\Phi_{s,t}f_0\rangle|^2=\EE\langle \psi\otimes \overline{\psi},\Phi_{s,t}f_0\otimes \overline{\Phi_{s,t}f_0}\rangle_{H\otimes H}\:.
\end{equation}
Pulling the expectation into the inner product, we thus aim to derive an equation satisfied by 
\begin{equation}
g(t):=\EE[\Phi_{s,t}f_0\otimes \overline{\Phi_{s,t}f_0}]\:.
\end{equation}
Proceeding in a purely formal way, we see that \(\overline{\Phi_{s,t}f_0}=\overline{f(t)}\) satisfies
\begin{equation}
\dd\overline{f(t)}-\sum_{k=1}^\infty i\overline{L_k f(t)}\circ \dd W^k_t=0\:,
\end{equation}
and hence, applying Itô's formula for tensor products (see Appendix~\ref{Ito formula}), we deduce that 
\begin{equation}
\dd(f(t)\otimes \overline{f(t)})=\dd M_t+\frac{1}{2}\sum_{k=1}^\infty(iL_k\otimes \Id-i \Id\otimes \overline{L_k})^2(f(t)\otimes \overline{f(t)})\:,
\end{equation}
where \(M_t\) is a martingale.
Thus, taking expectation, it follows that $g$ satisfies
\begin{equation}
\partial_t g =Lg\;,
\end{equation}
where 
\begin{equation}\label{eq:Loperator}
L=\frac{1}{2}\sum_k(iL_k\otimes \Id-i\Id\otimes \overline{L_k})^2
\end{equation}
is a symmetric, negative operator on $H\otimes H$. Continuing with the theme of brushing technical considerations under the rug and treating $L$ as if it was actually self-adjoint for now, the functional calculus implies that
\begin{equation}
\lim_{t\to \infty} \EE|\l \psi,\Phi_{s,t}f_0\r|^2 =\lim_{t\to \infty} \l \psi\otimes \overline{\psi},\e^{L(t-s)}\EE(f_0 \otimes \overline{f_0})\r_{H\otimes H} =
\langle \psi\otimes \overline{\psi},P\EE(f_0 \otimes \overline{f_0})\rangle_{H\otimes H}\:,
\end{equation}
where \(P\) is the orthogonal projection onto the kernel of \(L\). Therefore, upon classifying said kernel, Theorem~\ref{thm:sRAGE} follows readily. This will be carried out in Section~\ref{sec:RAGEthm}.

The idea of proving mixing for stochastic flows via the two point motion has found great success in recent years, starting from the seminal works \cite{DKK04,ALS09} and, in the context of stochastic fluid mechanics, with \cite{Bedrossian_Blumenthal_Punshon-Smith_2022,Blumenthal_CotiZelati_Gvalani_2023,Gess_Yaroslavtsev}. In all these cases, the authors use Markov chain techniques to prove the existence of a spectral gap for the generator of the two point motion. From there, the above considerations can be used to deduce exponential mixing, even uniform with respect to the diffusivity parameter $\nu$, see \cite{BBSP21,CIS24}. We also mention the works \cite{Rowan_2024,Navarro-Fernández_Seis_2025}, where the authors derive a spectral gap for the generator of the two-point motion using PDE techniques. At the level of generality of our setting, we cannot expect a spectral gap for the generator of the two-point motion to always exist. However, in order to deduce the non-quantitative decay rates we are after, it will suffice to show that the kernel of the generator of the two point motion is trivial.

On the other hand, when considering stochastic shear flows in Section~\ref{sec:stochshear}, we will indeed exhibit a quantitative spectral gap in order to prove Theorem~\ref{thm:shearEn}. In particular, upon taking the Fourier transform in $x$, the evolution of the two point motion will be given by the action of a Schrödinger operator with non-negative potential \(V(x,y)\). The proof will then consist of estimating the size of the smallest eigenvalue of this operator in the semiclassical limit, which via the trace theorem followed by a Borel--Cantelli type argument will then yield the desired result.

\section{A Stochastic RAGE Theorem}\label{sec:RAGEthm}
In this section, we prove Theorem~\ref{thm:sRAGE} in the general setting briefly described in \eqref{eq:abstrdiff} and \eqref{eq:InviscidAbstract}. 
Let $\mathcal{H}$ be a separable real Hilbert space, endowed with norm
$\|\cdot\|$ and inner product $\l\cdot,\cdot\r$.
We write $H$ for its complexification \(\mathcal{H}\otimes \mathbb{C}\), and consider a family of self-adjoint unbounded operators 
$$
L_k:\cD(L_k)\to H, \qquad k\in \NN\:,
$$
such that $\bigcap_k \cD(L_k)$ is a common core for the $L_k$ in $H$, and so that the sequence $\{\|L_kx\|\}_{k \in \mathbb{N}} \in \ell^1$, for all $x \in \bigcap_k \cD(L_k)$.
We work on a filtered probability space $(\Omega,\mathcal{F},\{\mathcal{F}_t\}_t,\mathbb{P})$
supporting a countable number of i.i.d.\ standard Brownian motions $\{W^k_t\}_{k \in \mathbb{N}}$, which generate the filtration $\{\mathcal{F}_t\}_t$.
We consider stochastic partial differential equations on \(H\) of the form
\eqref{eq:InviscidAbstract}, for which our definition of strong solutions differs slightly from that found in the literature (see e.g. \cite{Prato_Zabczyk_2014}), but will turn out to be natural for the setting we are interested in.
\begin{definition}
An adapted process \(f\in L^2(\Omega;L^2_\loc(0,\infty;H))\)\footnote{In the sense that
$f \restr [0,T]\in L^2(\Omega;L^2(0,T;H))$ for every $T > 0$.} is called a strong solution to \eqref{eq:InviscidAbstract} if the following conditions hold:

\begin{itemize}
\item[--] for each $t\geq0$, $f(t)$ takes values in \(\bigcap_{k}\cD(L_k^2)\) almost surely;
\item[--] for every $T>0$, one has
\begin{equation}
\label{eq:boundSol}
\EE \int_0^T \sum_k (\|L_k f(t)\|^2+\| L_k^2 f(t)\|)\,\dd t < \infty\;,
\end{equation}
\item[--] for every $t>0$, the identity
\begin{equation}
f(t)=f(0)-\sum_{k}\int_{0}^tiL_kf(s)\dd W^k_s+\frac{1}{2}\sum_{k}\int_{0}^t(iL_k)^2f(s)\dd s\;
\end{equation}
holds almost surely.
\end{itemize}
\end{definition}
With this definition at hand, we then make the following assumption regarding the existence and 
uniqueness of solutions to  \eqref{eq:InviscidAbstract}:
\begin{enumerate} [label=(A\arabic*), ref=(A\arabic*)]
\item\label{Assumption1}
There exists a norm-dense subspace  \(\cC\) of \(H\) so that, for any initial condition \(f_0 \in \cC\), there exists a unique strong solution \(f(t)\) to \eqref{eq:InviscidAbstract} with initial datum \(f_0\). Therefore, this solution is given by a solution operator \(\Phi_{s,t}:\cC \to H\), which is a random isometry, i.e. for any \(s,t\), it holds that \(\|\Phi_{s,t}f_0\|=\|f_0\|\). By continuity, we extend this map to all of \(L^2(\Omega,\cF_s,\PP,H)\), and call \(\Phi_{s,t}f_0\) the solution to \eqref{eq:InviscidAbstract}.
\end{enumerate}

The point of this assumption is that it enables us to make the formal computations from 
Section~\ref{sub:mainideas} rigorous. In particular, under our well-posedness assumption, the two-point 
correlation $g(t)=\mathbb{E}(f(t) \otimes \overline{f(t)})$ is well defined for any $f_0 \in \mathcal{C}$. 
Furthermore, it satisfies $\partial_t g(t)=L g(t)$, for any self-adjoint extension of the symmetric operator $L$. 
We show in Lemma \ref{Computation is rigorous} that this is in fact enough to conclude that $L$ is 
essentially self-adjoint when given an appropriate domain. We further remark that our well-posedness assumptions 
will be satisfied by most of the operators we care about. In particular, it holds in the case where the 
$iL_k$ are of the form $\sigma_k \cdot \nabla$.

\subsection{The Two-Point Motion}
We define $\sym$ as the closed \textit{real} subspace of \(H \otimes H\) generated 
by real linear combinations of \(\{x\otimes \overline{x}:x \in H\}\). Elements 
of $\sym$ are
canonically identified with self-adjoint Hilbert--Schmidt operators on $H$ via the map $T$ (see Lemma \ref{Hilbert--Schmidt isomorphism})
\begin{equation}
T_{x\otimes \overline{x}}(y):=\langle y,x\rangle  x\;,
\end{equation}
so that 
\begin{equation}
\label{tensor product inner product}\langle T_{x\otimes \overline{x}}(y),z\rangle =\langle x \otimes \overline{x},z\otimes \overline{y} \rangle,  \quad \langle y,T_{x\otimes \overline{x}}(z)\rangle =\langle y,x\rangle \overline{\langle \overline{x},\overline{z}\rangle }=\langle x\otimes \overline{x},z\otimes \overline{y}\rangle\:.
\end{equation}
By linearity and continuity, the identities in \eqref{tensor product inner product} may be extended to any $\zeta \in \sym$.
We will see in Lemma~\ref{intersection of kernels} that the restriction of the operator $L$ to $\sym$ is 
essentially self-adjoint. Our strategy in classifying $\Ker(L)$ will be to view its elements as self-adjoint Hilbert--Schmidt operators. This allows us to apply the extremely rich spectral theory of compact self-adjoint operators to obtain a complete classification of $\Ker(L)$.

We begin by stating the following result, which essentially provides a rigorous justification of the formal computation we did in the previous section.
\begin{lemma}\label{Computation is rigorous}
Let \(\Phi_{s,t}\) satisfy~\ref{Assumption1}. Then the operator $L$, with domain $\mathcal{D}(L)= \Span\{\mathbb{E}(f\otimes \overline{f}):f \in \mathcal{C}\}$  is essentially self-adjoint on \(\sym\), negative, and has kernel given by 
\begin{equation}
\label{intersection of kernels}
\Ker(L)=\bigcap_{k=1}^\infty\Ker\big(L_k\otimes \Id-\Id\otimes \overline{L_k}\big)\:.
\end{equation}
Furthermore, identifying \(L\) with its closure, it holds that  \(g(t):=\EE(\Phi_{s,t}f_0\otimes \overline{\Phi_{s,t}f_0})\) satisfies
\begin{equation}
g(t)=\e^{Lt}\EE(f_0\otimes \overline{f_0})\:,
\end{equation}
for any \(f_0 \in L^2(\Omega,\cF_s,\PP,H)\).
\end{lemma}
The proof of this result is the subject of the Appendix~\ref{app:techno}. 
\begin{remark}
The expression \eqref{intersection of kernels} for the kernel is more subtle than it
may appear. Whilst one certainly would expect the kernel of a sum of negative operators to be the intersection of their respective kernels, we emphasize that this property is \emph{not} stable under general extensions of unbounded operators. In particular, in view of the classical result of Krein (see e.g. \cite{Fucci_Gesztesy_Kirsten_Littlejohn_Nichols_Stanfil_2022} for a recent survey), 
any positive self-adjoint extension $\Tilde{L}$ of $-L$ satisfies 
\begin{equation}
L_K \leq \Tilde{L} \leq L_F
\end{equation}
where $L_K$ denotes the Krein--von Neumann extension, and $L_F$ the Friedrichs extension. While the Friedrichs extension enjoys property \eqref{intersection of kernels} for its kernels, the kernel of the Krein--von Neumann extension satisfies $\Ker(L_K)=\Ker(L^*)$ (see for example \cite{Fucci_Gesztesy_Kirsten_Littlejohn_Nichols_Stanfil_2022}), which may in principle be a lot larger. By imposing assumption \ref{Assumption1} and concluding that the operator $L$ is essentially self-adjoint, we entirely circumvent this technical problem.
\end{remark}

\begin{remark}The identification of tensors with Hilbert--Schmidt operators is chosen so that it aligns precisely with the natural identification of functions $f \in L^2(\mathbb{T}^{2d})$ with Hilbert--Schmidt integral operators $g \mapsto \int_{\mathbb{T}^d} f(x,y)g(y)\dd y=\langle f(x,\cdot),\overline{g}(\cdot)\rangle_{L^2(\mathbb{T}^d)}$. Therefore, when $H\subset L^2(\TT^d)$, the space $\sym\subset H\otimes H$ can also be identified with 
square integrable kernels satisfying  the symmetry condition \(h(x,y)=\overline{h(y,x)}\).
\end{remark}

\subsection{The Kernel of \TitleEquation LL}
With these preliminaries at hand, we may relate the kernel of $L$ to the spectrum of the operators $L_k$. In what follows, we denote by $\sigma_p(L_k)$ the point spectrum of \(L_k\). 
We have the following result.
\begin{lemma}\label{lem:SpecLk}
The kernel of \(L\) on \(\sym\) is given by
\begin{equation}\label{eq:KernelL}
  \Ker(L)= \bigcap_{k=1}^\infty \overline{\Span\left\{\zeta\otimes \overline{\zeta}:
L_k\zeta=\lambda \zeta , \lambda \in \sigma_p(L_k)\right\}}\:.
\end{equation}
\end{lemma}
\begin{proof}
Writing \(\widetilde{L}_k:=(iL_k\otimes \Id -i\Id \otimes \overline{L_k})^2\) as a self-adjoint
operator on \(\sym\), the claim follows from Lemma~\ref{Computation is rigorous} if we can show that
\begin{equation}
\sym\cap\Ker \widetilde{L}_k=\overline{\Span\left\{\zeta\otimes \overline{\zeta}:
L_k\zeta=\lambda \zeta , \lambda \in \sigma_p(L_k)\right\}}\;.
\end{equation}
We now consider $k$ as fixed. By the spectral theorem, we can find a standard 
probability space $(X_k,\mu_k)$ and a function $\lambda_k \colon X_k \to \RR$ such that 
$H$ is unitarily equivalent to $L^2(X_k,\mu_k)$ and, under this equivalence, $L_k$ is 
simply the multiplication operator by $\lambda_k$. Note that the complex structure of $H$ is
in general \textit{not} mapped to the standard complex structure on $L^2(X_k,\mu_k)$. However, we can
also apply the spectral theorem to $\overline{L_k}$, yielding a unitary equivalence 
$H \approx L^2(\bar X_k,\bar \mu_k)$ such that $\overline{L_k}$ is equivalent to multiplication by a map
$\bar \lambda_k \colon \bar X_k \to \RR$, as well as an anti-unitary map $u \mapsto \bar u$ from
$L^2(X_k,\mu_k)$ to $L^2(\bar X_k,\bar \mu_k)$ with the property that 
$\overline{\lambda_k \cdot u} = \bar \lambda_k \cdot \bar u$.

We then have $H \otimes H \approx L^2(X_k \times \bar X_k,\mu_k \otimes \bar \mu_k)$ and, under this equivalence,
\begin{equation}
\bigl(\widetilde{L}_k u\bigr)(x,y) = (\bar\lambda_k(y) - \lambda_k(x))^2 u(x,y)\;,
\end{equation}
so that $\Ker(\widetilde{L}_k)$ consists of those functions supported on $\Delta = \{(x,y) \,:\, \bar\lambda_k(y) = \lambda_k(x)\}$. 
Since $\lambda_k$ and $\bar \lambda_k$ are independent and independent random variables can only be equal with positive probability if their laws 
have at least one atom with positive mass, we can write 
\begin{equation}
\Delta \simeq \bigcup_{\lambda \in \sigma_p(L_k)} X_k^{(\lambda)}\times \bar X_k^{(\lambda)}\;,
\end{equation}
where $\sigma_p(L_k) = \{\lambda \,:\, (\lambda_k^\sharp \mu_k)(\{\lambda\}) > 0\}$, 
$X_k^{(\lambda)} = \lambda_k^{-1}(\{\lambda\})$, $\bar X_k^{(\lambda)} = \bar \lambda_k^{-1}(\{\lambda\})$, and $\simeq$ denotes equivalence up to null sets.
Note that $\sigma_p(L_k)$ is necessarily at most countable since the measure $\mu_k$ is finite.

Writing $H_k^{(\lambda)}\subset H$ for the subspace corresponding to  $L^2(X_k^{(\lambda)},\mu_k) \subset L^2(X_k,\mu_k)$
and similarly for $\bar H_k^{(\lambda)}\subset H$, we conclude that 
\begin{equation}
\label{eq:kerLtilde}
\Ker \widetilde{L}_k = \cl \Bigl(\bigoplus_{\lambda \in \sigma_p(L_k)} H_k^{(\lambda)} \otimes \bar H_k^{(\lambda)}\Bigr)\;.
\end{equation}
Since $H_k^{(\lambda)}$ is nothing but the eigenspace of $L_k$ with eigenvalue $\lambda$, we know
that complex conjugation maps $H_k^{(\lambda)}$ to $\bar H_k^{(\lambda)}$. The claim then follows
from the definition of $\sym$.
\end{proof}

\begin{lemma}
\label{eigs are finite dim}
Let \(\zeta\in \Ker(L)\), and fix any $k \in \mathbb{N}$. Then, any eigenvector of \(T_\zeta\) corresponding to a non-zero eigenvalue is a finite linear combination of eigenfunctions of \(L_k\).
\end{lemma}
\begin{proof}
Since one also has $\zeta \in \Ker \widetilde L_k$ by Lemma~\ref{Computation is rigorous}, \eqref{eq:kerLtilde} yields a decomposition 
$\zeta=\sum_{\lambda \in \sigma_p(L_k)} \zeta_\lambda$,
where \(\zeta_\lambda \in H_k^{(\lambda)}\otimes \bar H_k^{(\lambda)}\).
The operator $T_\zeta$ is therefore block diagonal under the direct sum decomposition
\begin{equation}
H = H^{\cont} \oplus \bigoplus_{\lambda \in \sigma_p(L_k)}H_k^{(\lambda)}\;,
\end{equation}
and such that $H^{\cont} \subset \Ker T_\zeta$.
Since $T_\zeta$ is Hilbert--Schmidt, given any value $\gamma \neq 0$, there can be at most finitely
many blocks with an eigenvalue $\gamma$, thus concluding the proof.
 \end{proof}
With this preliminary result in place, we can give the following characterisation of the kernel of \(L\) in the next lemma.
\begin{lemma}\label{lem:KerL}
The kernel of L, viewed as an operator on \(\sym\), is precisely the set of elements \(\zeta \in \sym\), such that the following holds: there exists a collection \(\{V_j\}_{j \in \mathbb{N}}\) of mutually orthogonal, finite dimensional subspaces of $H$ that are invariant for all of the \(L_k\), and a collection of numbers \(\gamma_j\in \mathbb{R}\) so that
\begin{equation}
\label{eq:kernelExpression}
T_\zeta = \sum_{j \in \mathbb{N}}\gamma_j P_j\:,
\end{equation}
where $P_j$ is the orthogonal projection onto $V_j$ (viewed as an element of $\sym$).
In particular, writing \(Q\) for the orthogonal projection from $H$ onto  \(H_\inv\), it holds that for any \(s\geq 0\), \(\psi \in L^2\), $\phi$ \(\cF_s\)-measurable,
\begin{equation}
\EE|\langle \psi,(1-Q)\Phi_{s,t}\phi\rangle|^2 \to 0\:.
\end{equation}
\end{lemma}
\begin{proof}
We first prove the inclusion that if \(\zeta=\sum_{j \in \mathbb{N}}\gamma_j\sum_{\ell=1}^{m_j}\phi^j_\ell \otimes\overline{\phi^j_\ell } \), where $\phi^j_\ell$ are an orthonormal basis for $V_j$, then \(\zeta\in \Ker(L)\). In fact, we must simply show that \(L_k\otimes \Id -\Id\otimes \overline{L_k}\) applied to all of the terms in the direct sum vanishes. 
Indeed, for fixed $j$, we have, using the invariance of the space under \(L_k\): 
\begin{equation}
(L_k\otimes \Id -\Id\otimes \overline{L_k})\sum_{\ell=1}^{m_j}\phi^j_\ell\otimes\overline{\phi^j_\ell}=\sum_{\ell,m=1}^{m_j}\langle \phi^j_m ,L_k \phi_\ell^j\rangle \phi_m^j\otimes \overline{\phi_\ell^j}-\overline{\langle \phi^j_m,L_k\phi_\ell^j\rangle}\phi_\ell^j\otimes \overline{\phi_m^j}\:.
\end{equation}
Relabeling \(\ell \to m\), \(m \to \ell\) in the second term, and using properties of inner products, this is 
\begin{equation}
\sum_{\ell,m=1}^{m_j}\left(\langle \phi^j_m ,L_k \phi_\ell^j\rangle -\langle L_k \phi^j_m ,\phi^j_\ell\rangle \right)\phi^j_m\otimes \overline{\phi_\ell^j}\:.
\end{equation}
This vanishes by symmetry of the operator \(L_k\), showing the inclusion.
To prove the converse, we note first, that by Lemma~\ref{Hilbert--Schmidt isomorphism} we can identify the space \(\sym \) with the space of self-adjoint Hilbert--Schmidt operators from \(H\) to \(H\). We now claim that given an element \(\zeta \in \Ker(L)\), the eigenspaces of its associated Hilbert--Schmidt operator $T_\zeta$ corresponding to non-zero eigenvalues will be finite dimensional invariant subspaces for all the \(L_k\). Indeed, we know that if \(\zeta \in \Ker(L)\), then \(\zeta \in \Ker(L_k\otimes \Id-\Id\otimes \overline{L_k})\), for all \(k\). By density of $\mathcal{D}(L_k)$, \(\psi \) is an eigenfunction of $T_\zeta$ with eigenvalue \(\gamma\), if and only if for any \(\phi \in \cD(L_k)\), we have
\begin{equation}
\langle T_\zeta\psi, \phi\rangle =\gamma \langle \psi, \phi\rangle \;.
\end{equation}
Therefore, letting $\psi$ be an eigenfunction of $T_\zeta$, and using repeatedly the relation \eqref{tensor product inner product}, it holds by \eqref{tensor product inner product} that (note that \(\psi \in \cD(L_k)\) by Lemma~\ref{eigs are finite dim})
\begin{align}
\langle T_\zeta(L_k \psi),\phi \rangle_H &=\langle \zeta,  \phi \otimes \overline{L_k\psi}\rangle_{H\otimes H} = \langle ( \Id\otimes \overline{L_k})\zeta,\phi \otimes \overline{\psi}\rangle_{H\otimes H}=\langle (L_k \otimes \Id)\zeta,\phi \otimes \overline{\psi} \rangle_{H \otimes H} \notag \\
&=\langle \zeta , L_k\phi \otimes \overline{\psi}\rangle_{H \otimes H} =\langle T_{\zeta}(\psi),L_k\phi\rangle_H =\gamma \langle  \psi,L_k\phi\rangle_H =\gamma \langle L_k\psi,\phi\rangle_H\:.
\end{align}
Hence, it holds that \(L_k\psi\) is also an eigenfunction with eigenvalue \(\gamma\), 
whence we conclude that the eigenspaces of $T_\zeta$ are invariant for all the \(L_k\). Now,  since   \(T_\zeta\) is compact and self-adjoint,  its eigenspaces corresponding to non-zero eigenvalues are finite dimensional, and mutually orthogonal. By the spectral theorem, we may write \(T_\zeta =\sum_{j}\gamma_j P_j\), where $P_j$ is the orthogonal projection onto the eigenspace of $T_\zeta$ corresponding to the eigenvalue $\gamma_j$. We now easily observe the final part of the result, by noting that (recall $\pi(\{0\})$ is the orthogonal projection onto $\Ker(L)$)
 \begin{align}
\lim_{t \to \infty}\mathbb{E}|\langle\psi,(1-Q)\Phi_{s,t}\phi\rangle|^2 &=\lim_{t \to \infty}\langle (1-Q)\psi \otimes \overline{(1-Q)\psi}, \e^{Lt}\mathbb{E}[\phi \otimes \overline{\phi}]\rangle_{H\otimes H}\\\
&=\langle (1-Q)\psi \otimes \overline{(1-Q)\psi},\pi(\{0\})\mathbb{E}[\phi\otimes \overline{\phi}]\rangle_{H \otimes H}\\
&=\langle (1-Q)\psi,T_{\zeta} (1-Q)\psi\rangle_{H}\:,
\end{align}
where $\zeta =\pi(\{0\})\mathbb{E}[\phi\otimes \overline{\phi}]$, and we have used relation \eqref{tensor product inner product} for the last equality. Since $\zeta \in \Ker(L)$, by the expression we obtained for elements in the kernel, $T_{\zeta}$ maps into the space $H_\inv$, which is left invariant by $Q$, and therefore
\begin{equation}
\langle (1-Q)\psi,T_{\zeta} (1-Q)\psi\rangle_{H}=0\;,
\end{equation}
showing the result.
\end{proof}
We have now proved all the preliminaries, and can thus conclude with the proof of Theorem~\ref{thm:sRAGE}. In fact, in addition to Theorem~\ref{thm:sRAGE}, we shall be able to prove the following slightly stronger statement.
\begin{proposition}\label{proposition:sRAGE}
Under the assumptions of Theorem \ref{thm:sRAGE}, there furthermore holds the following. Let \(\sym\) be the closed subspace of \(H \otimes H\) generated by \(\{x\otimes \overline{x}:x \in H\}\), and let \(K_0\) be the orthogonal projection in \(\sym\) onto \(\Ker(L)\), where $L=-\sum_{k}(L_k\otimes \Id-\Id\otimes \overline{L_k})^2$. If $K:H\to H$ 
is a finite rank linear operator, and, for some $S\subset L^2(\Omega, \cF_\infty,\PP,H)$, the set \(\cK =\{\EE(\phi \otimes \overline{\phi}):\phi \in S\} \) is pre-compact, then for $\phi \in \mathcal{F}_s$
\begin{equation}\label{eq:unifconv}
\lim_{T \to\infty}\frac{1}{T}\int_{s}^{s+T} \EE\|K \Phi_{s,t}\phi\|^2\dd t = \EE\tr (K\otimes \bar K) K_0(\phi \otimes \overline \phi)\:,
\end{equation}
uniformly over \(\phi \in S\), where we view $(K\otimes \overline{K})K_0(\phi \otimes \overline \phi)$ as a random trace class operator $H \to H$. (Note that it is indeed of trace class, since $K$ is finite rank).
\end{proposition}

\begin{proof}[Proof of Theorem~\ref{thm:sRAGE} and Proposition~\ref{proposition:sRAGE}]
In view of the previous Lemma~\ref{lem:KerL}, we have already proven the statement of Theorem~\ref{thm:sRAGE} when \(K\) is finite rank. The compact case is the same as for the deterministic RAGE theorem, so we omit it.

The commutativity claim follows upon noting that if \(V\) is a closed subspace which is invariant for all the \(L_k\), and \(P_V \) is the orthogonal projection onto \(V\), then \(P_V\) commutes with all the \(L_k\). To see that, note that  \(P_VL_kf=P_VL_k(P_Vf+(1-P_V)f)=L_kP_Vf+P_VL_k(1-P_V)f\). Now, for any \(v \in H_\inv\), we have 
\begin{equation}
\langle v, P_VL_k(1-P_V)f\rangle =\langle L_k P_Vv,(1-P_V)f\rangle =\langle P_VL_kP_Vv,(1-P_V)f\rangle =0 \:,
\end{equation}
where we have used the symmetry of the \(P_V, L_k\), and the invariance of the range of \(P_V\) under \(L_k\). It thus follows that \(P_VL_k(1-P_V)=0\), and so \(P_VL_kf=L_kP_Vf\).

To show the uniform convergence claim of Proposition~\ref{proposition:sRAGE}, note that the pointwise convergence claim follows readily by our previous arguments, and so we thus just need to establish equicontinuity in a suitable sense. Expanding the expression in \eqref{eq:unifconv}, we get
\begin{align}
F(T)&\eqdef\frac{1}{T}\int_{s}^{s+T} \sum_{i=1}^n|\alpha_i|^2\langle \psi_i\otimes \overline{\psi_i},\e^{L(t-s)}\EE[\phi\otimes \phi]\rangle \dd t \notag\\
&=\frac{1}{T}\int_{0}^{T} \sum_{i=1}^n|\alpha_i|^2\langle \psi_i\otimes \overline{\psi_i},\e^{Lt}\EE[\phi\otimes \phi]\rangle \dd t\:.
\end{align}
Thus, we have
\begin{equation}
|F(T)(\phi_1)-F(T)(\phi_2)|\leq \sum_{i=1}^n|\alpha_i|^2 \|\psi_i\otimes \overline{\psi_i}\|\frac{1}{T}\int_{0}^T \left\|\EE\left[\phi_1(t)\otimes \overline{\phi_1(t)}-\phi_2(t)\otimes \overline{\phi_2(t)}\right]\right\|\dd t \:.
\end{equation}
Now, the function 
$$
f(t)\eqdef\EE\left[\phi_1(t)\otimes \overline{\phi_1(t)}-\phi_2(t)\otimes \overline{\phi_2(t)}\right]
$$
satisfies \(f(t) =e^{Lt}f(0)\), and hence
$$
\|f(t)\|\leq\|f(0)\|=  \|\EE(\phi_1\otimes \overline{\phi_1})-\EE(\phi_2\otimes \overline{\phi_2})\|, \qquad \forall t\geq 0\:.
$$
But now, note that we may view the sequence $F(T)$ as continuous functions $F(T):\sym\to \mathbb{R}$, and therefore we have just shown that the sequence is in fact equicontinuous, and so uniform convergence on the compact set $\cl(\cK)$ follows by Arzela--Ascoli.
This concludes the proof.
\end{proof}

We finally give a result which gives a way to decompose the space \(H_\inv\) into a direct sum of finite dimensional invariant subspaces, and therefore to prove the sharpness of our Stochastic RAGE theorem:
\begin{lemma}\label{lem:Zorn}
The set \(H_\inv\) may be decomposed as
\begin{equation}
\label{Zorn}
H_\inv=\cl \left( \bigoplus_{n=1}^\infty V_n \right)\:,
\end{equation}
where the \(V_n\) are mutually orthogonal finite dimensional invariant subspaces for all the \(L_k\) which have no further non-trivial subspaces that are invariant for all the \(L_k\).
\end{lemma}
\begin{proof}
The proof is a standard application of Zorn's Lemma. 
Let now \(P\) be the set of all subsets of the form \(\{V_1,V_2,\dots\}\), so that the \(V_i\) are finite dimensional invariant subspaces of the \(L_k\), that have no non-trivial invariant subspaces of the \(L_k\) and are mutually orthogonal. We equip this space with the partial order given by inclusion. Let now \(\{S_i\}_{i \in I}\) be a chain in \(P\). Then \(\bigcup_{i \in I}S_i\) is clearly an upper bound for the chain. Hence, by Zorn's Lemma we get a maximal element \(S=\{V_1,V_2,\dots\}\) for $P$ (note that $S$ is a countable collection, since $H$ is assumed separable). We now claim that \(H_\inv =\cl(\bigoplus_{j}V_j)\). Indeed, suppose not. Then, since \(H_\inv\) is the smallest closed vector space containing all finite dimensional invariant subspaces, and certainly \(\cl(\bigoplus_{j}V_j)\) is closed in the $\|\cdot \|_{H}$ topology, by orthogonality of the \(V_j\), there must exists a finite dimensional invariant subspace \(V\) which is not completely contained in \(\cl(\bigoplus_{j}V_j)\). Let \(Q\) be the orthogonal projection onto \(\cl(\bigoplus_{j}V_j)\). Then it follows that \((1-Q)V\) is a non-trivial, finite dimensional vector space. In fact, by the proof of the RAGE theorem, it is invariant for all the \(L_k\), since \(T\) and the \(L_k\) commute, so that for any \(v \in V\), \(L_k(1-Q)v=(1-Q)L_kv \in (1-Q)V\), since \(V\) is invariant for \(L_k\). Thus, taking a minimal invariant subspace \(\widetilde{V}\) of \((1-Q)V\), the set \(\{\widetilde{V}\}\) is orthogonal to every element in \(S\), and contains no non-trivial invariant subspaces, so that \(S\cup \{\widetilde{V}\}\) is a well defined element in \(P\), contradicting the maximality of \(S\).
\end{proof}

\begin{remark}Note that this implies that our version of the stochastic RAGE theorem is essentially sharp. Indeed, suppose that \(Qf\neq 0\). Then, writing \(P_j\) for the projection onto \(V_j\), Lemma~\ref{lem:Zorn}
implies that there exists at least one \(j\) so that \(P_jf \neq 0\). Let now \(K =P_j\), which is compact since \(V_j\) is finite dimensional. Then, it follows that 
\begin{equation}
\frac{1}{T}\int_{0}^T\EE\|K\Phi_{0,t}f\|^2\dd t =\frac{1}{T}\int_{0}^T\EE\|\Phi_{0,t}f\|^2\dd t\:.
\end{equation}
But we know that the SPDE conserves the norm, so that \(\EE\|\Phi_{0,t}f\|^2\dd t =\EE\|P_jf\|>0\), for all \(t\), and so \(\frac{1}{T}\int_{0}^T\EE\|K\Phi_{0,t}f\|^2\dd t \nrightarrow 0\).
\end{remark}
Finally, we remark on an elementary consequence of our analysis to the existence of invariant measures for the transport SPDE.
\begin{corollary}[Invariant Measures]
\label{invariant measures}
Consider the linear SPDE 
\begin{equation}
\label{eq:SPDE}
\dd f_t+\sum_{k}iL_k f_t \circ \dd W^k_t=0
\end{equation}
on the Hilbert space $H$.
This SPDE admits a non-trivial invariant measure if and only if there exists a finite dimensional subspace invariant under the action of the $iL_k$. Furthermore, every invariant measure $\mu$ of \eqref{eq:SPDE} 
satisfies $\supp(\mu) \subset H_\inv$.
\end{corollary}
\begin{proof}
Suppose the SPDE admits an invariant measure $\mu_0\neq \delta_{0}$. Since
\eqref{eq:SPDE} preserves the norm in $H$, we can then find a non-zero ergodic invariant probability
measure $\Tilde{\mu}$ such that $\supp(\Tilde{\mu})\subset B_{1}$, the unit ball. In particular, for any $y \in H$, the map $x \mapsto |\langle y,x\rangle|^2$ is in $L^1(\Tilde{\mu})$.
Therefore, it follows by Birkhoff's ergodic theorem, that for $\tilde{\mu}$-a.e.  initial condition $f_0$, there holds that
\begin{equation}
\frac{1}{T}\int_0^T|\langle \Phi_{0,t}f_0,x\rangle |^2 \dd t \to \int_{H}|\langle y, x \rangle |^2 \dd\Tilde{\mu}(y)
\end{equation}
almost surely. However, taking expectations on both sides and employing dominated convergence, we see that 
\begin{equation}
\frac{1}{T}\int_{0}^T \mathbb{E}|\langle \Phi_{0,t}f_0,x\rangle |^2 \dd t \to 0
\end{equation}
as $T \to \infty$. Thus, we see that $\supp(\Tilde{\mu})\subset \mathrm{span}\{x\}^{\perp}$, and hence it follows that $\supp(\Tilde{\mu})=\{0\}$ since $x$ was arbitrary. To show the converse statement, suppose that there exists a finite dimensional subspace $V \subset H$ so that $iL_k V\subset V$, for all $k \in \mathbb{N}$. Restricting the SPDE to $V$, and writing elements of $V$ as vectors $X$ in $\mathbb{R}^m$, where $m=\mathrm{dim}(V)$, it follows that the SPDE becomes 
\begin{equation}
\label{finite dimensional system}
\dd X_t+\sum_{k \geq 0} A_k X_t \circ \dd W^k_t=0 \:,
\end{equation}
where the $A_k$ are a collection of antisymmetric matrices. Hence, it suffices to show the existence of a non-trivial invariant measure of \eqref{finite dimensional system} on $\mathbb{R}^m$. This in turn follows from
the Krylov--Bogolyubov theorem, since \eqref{finite dimensional system} is Feller and preserves the unit ball.
\end{proof}
\section{Enhanced Dissipation}\label{sec:enhanceddiss}
Having the stochastic RAGE theorem at our disposal, we are in the position of proving the enhanced dissipation result in Theorem~\ref{main result}. 
In addition to the abstract framework introduced in Section~\ref{sec:RAGEthm}, we let $A$ be a strictly positive self-adjoint unbounded linear operator
\begin{align}
A:\cD(A)\to H
\end{align}
such that the domain $\cD(A)$ is compactly embedded in $H$.
Thus, $A$ possesses a
strictly positive sequence of discrete eigenvalues $\{\lambda_\ell\}_{\ell \in \NN}$\:, 
\begin{equation}
\begin{cases}
0<\lambda_1\leq \lambda_2 \leq \ldots,\\
\lambda_\ell\to \infty \quad\text{for} \quad \ell \to \infty,
\end{cases}
\end{equation}
with associated eigenvectors $\{e_\ell\}_{\ell\in\NN}$, that form an orthonormal basis for $H$. Any element $\phi \in H$ can be therefore written uniquely as 
\begin{align}
\phi=\sum_{\ell\geq 1} \phi_\ell e_\ell, \qquad \phi_\ell=\l \phi,e_\ell\r \:.
\end{align}
For any $s\in \RR$, we can then define the scale of Hilbert spaces $H^s=\cD(A^\frac{s}{2})$, 
with norm 
\begin{align}
\|\phi\|^2_{H^s}=\sum_{\ell \geq 1} \lambda_\ell^{s}|\phi_\ell|^2.
\end{align}
(When $s < 0$, these spaces are obtained by completing $H$ under the corresponding norm.)
In addition, if one defines \(\widetilde{A}=A^\frac{1}{2}\otimes \Id+\Id\otimes A^\frac{1}{2}\), then  
Lemma~\ref{self-adjointtensor} implies that \(\widetilde{A}\) also has compact resolvent, and one can again define a scale of Hilbert spaces as above, which we denote $H^s \otimes H^s$.

We now consider the dissipative abstract evolution equation \eqref{eq:abstrdiff}. Further to the assumptions~\ref{Assumption1} from before, we place the following assumptions on the inviscid equation \eqref{eq:InviscidAbstract}:

\begin{enumerate} [label=(A\arabic*), ref=(A\arabic*)]
\setcounter{enumi}{1}
    \item\label{Assumption2} There exists a function \(B \in L^1_\loc(0,\infty)\) so that for any \(f_0 \in H^1\), the corresponding solution $f$ satisfies 
    \begin{equation}
    \EE\|f(t)\|_{H^1}^2 \leq B(t)\EE\|f_0\|_{H^1}^2\:,
    \end{equation}
    for every $t\geq0$.
    \item \label{Assumption3}  For any \(f_0 \in H\), the solution to the viscous problem \eqref{eq:abstrdiff} with initial data \(f_0\) lies in the dense set \(\cC\) of assumption~\ref{Assumption1} for any \(t>0\).
\end{enumerate}
The result of Theorem~\ref{main result} in this general setting reads as follows.

\begin{theorem}\label{Main Abstract Result}
Assume~\ref{Assumption1}--\ref{Assumption3}. The stochastic advection-diffusion equation \eqref{eq:abstrdiff} is dissipation enhancing if and only if there does not exist a non-trivial, finite dimensional subspace \(V\subset H^1\) which is invariant for all the \(iL_k\).
\end{theorem}

\begin{remark}\label{rmk:concrete}
In the case \(L_k=0\) for every $k\geq 2$, the condition of Theorem~\ref{Main Abstract Result} is equivalent to the absence of \(H^1\) eigenfunctions of \(L_1\), thus recovering the condition in \cite{Constantin_Kiselev_Ryzhik_Zlatos_2008}. Moreover, Theorem~\ref{main result} corresponds  to the case
$$
H=\left\{\phi\in L^2(\TT^d): \int_{\TT^d} \phi(x)\dd x=0\right\}\:,
$$
with $i L_k=\sigma_k\cdot \nabla $ and $A=-\Delta$ with domain $\cD(A)=\dot{H}^2(\TT^d)$.
\end{remark}

\subsection{Uniform Growth of the Inviscid Equation}
There are two main ingredients we will need in order to prove this result. The first is the generalisation of the RAGE theorem in Theorem~\ref{thm:sRAGE}, whilst the second is a uniform lower bound on the \(H^1\) growth of solutions to \eqref{eq:InviscidAbstract}. Recall that we denote by $K_0$ the $\sym$ orthogonal projection onto the kernel of $L$.
\begin{lemma}
\label{pure point growth}
Assume that there does not exist any non-trivial finite dimensional subspace of \( H^1\) which is invariant for all the $L_k$'s, and let $K_0$ be the $\sym$ orthogonal projection onto the kernel of $L$. Suppose that there exists $\cS \subset L^2(\Omega,\cF,\PP;H)$ so that the image $\cK$ of $\cS$ under the map $f \mapsto \mathbb{E}(f\otimes \overline{f})$ is pre-compact in $\sym$, and further that \(\inf_{f \in S}\|K_0\EE(f\otimes \overline{f})\|\geq a\), for some \(a>0\). Then, given \(B_0>0\), one can find \(N_0,T_0>0\), so that uniformly for all $s \geq 0$, for all \(N\geq N_0\), \(T\geq T_0\), one has
\begin{equation}
\frac{1}{T}\int_{s}^{s+T} \mathbb{E}\|P_N \Phi_{s,t}f\|_{H^1}^2\dd t \geq B_0, \qquad \forall f \in \cS \cap L^2(\Omega,\cF_s,\PP,H)\:,
\end{equation}
where \(P_N\) is the orthogonal projection onto the first \(N\) eigenfunctions of \(A\).
\end{lemma}
\begin{proof}
We first claim that given \(B_0>0\), one can pick \(N\in\NN\) so that for any \(n\geq N\), \(f \in \cS \cap (\cup_{s \geq 0}L^2(\Omega,\cF_s,\PP,H))\),
\begin{equation}\label{eq:kernelgrowth}
\mathbb{E} \tr(P_nAK_0 (f \otimes \bar f)) \geq 2B_0\;.
\end{equation}
Suppose not. Then, there exists a sequence $n_j \to \infty$, and corresponding elements \(\{f_{n_j}\} \subset \cS \cap L^2(\Omega,\cF_{s_{n_j}},\PP,H)\) so that 
\begin{equation}\label{eq:sumcontra}
\mathbb{E} \tr(P_{n_j} A K_0 (f_{n_j} \otimes \overline{f_{n_j}}))< 2B_0\;.
\end{equation}
By pre-compactness of $\cK$, modulo passing to a subsequence, one can find $ v \in \cl(\cK)$ such that \(\EE(f_{n_j}\otimes \overline{f_{n_j}}) \to v\).
Note that, for any $n_j\in \NN$, 
\begin{equation}
\langle K_0\EE(f_{n_j}\otimes \overline{f_{n_j}}),e_\ell \otimes \overline{e_\ell}\rangle =\lim_{t \to \infty}\mathbb{E}|\langle e_\ell,\Phi_{s_{n_j},t}f_{n_j}\rangle |^2 \geq 0\:,
\end{equation}
and hence, for fixed \(N\), we deduce that
\begin{align}
\tr(P_N AK_0 v)
&= \lim_{j \to \infty}\sum_{\ell=1}^N\lambda_\ell \l K_0\EE(f_{n_j}\otimes \overline{f_{n_j}}),e_\ell \otimes \overline{e_\ell}\r\\ 
&\leq \liminf_{j \to \infty} \mathbb{E} \tr(P_{n_j} A K_0 (f_{n_j} \otimes \overline{f_{n_j}})) \leq  2B_0\:.
\end{align}
Since \(K_0v \in \Ker(L)\) and \(K_0\) preserves real and imaginary parts, we may write 
\begin{equation}
 K_0v=\sum_{i}\alpha_i \sum_{m=1}^{m_i} \phi_m^i \otimes\overline{\phi_m^i} \;,
\end{equation}
where \(\{\phi_m^i\}_{m=1}^{m_i}\) is an orthonormal basis of the eigenspace corresponding to the eigenvalue \(\alpha_i\) of the self-adjoint Hilbert--Schmidt operator \(K_0v\). We claim that all the \(\alpha_i\) are non-negative. Indeed, by continuity, \(K_0v\) is the limit in \(\sym\) of \(K_0\EE(f_{n_j} \otimes \overline{f_{n_j}})\). Then
\begin{equation*}
\alpha_i=\langle K_0 v, \phi^i_m \otimes \overline{\phi^i_m}\rangle_{H \otimes H}= \lim_{j \to \infty}\langle K_0\EE(f_{n_j}\otimes \overline{f_{n_j}}),\phi_m^i \otimes\overline{\phi_m^i} \rangle_{H \otimes H} =  \lim_{j \to \infty}\lim_{t \to \infty} \EE|\langle  \phi_m^i, \Phi_{s_{n_j},t}f_{n_j}\rangle|^2 \geq 0\;.
\end{equation*}
Hence, \(\alpha_i \geq 0\), for all \(i\). Further, as \(\|K_0v\| \geq a\), there exists at least one strictly positive eigenvalue, say $\alpha_{i_0}>0$. Thus, for any \(N\)  we have 
\begin{equation}
\sum_{\ell=1}^N \lambda_\ell \sum_{i} \alpha_i \sum_{m=1}^{m_i} \langle e_\ell\otimes \overline{e_\ell},\phi_m^i \otimes \overline{\phi_m^i}\rangle \leq 2B_0\;.
\end{equation}
In particular,   
\begin{equation}
\alpha_{i_0} \sum_{m=1}^{m_{i_0}}\sum_{\ell=1}^N\lambda_\ell  |\langle e_\ell,\phi^{i_0}_m\rangle|^2\leq 2B_0 \:,
\end{equation}
for any \(N\). Taking \(N \to \infty\) yields that \(\|\phi_m^{i_0}\|^2_{H^1}\leq 2\alpha_{\ell_0}^{-1}B_0\) for \(m=1,\dots, m_{i_0}\), and so it follows that \(\mathrm{span}\{\phi^{i_0}_m\}_{m=1}^{m_{i_0}}\subset H^1\). But this is a contradiction since we know this to be a finite dimensional invariant subspace of all of the \(L_k\), and hence it is not entirely contained in \(H^1\) by  assumption. Thus one can indeed pick \(N\) so that \eqref{eq:kernelgrowth} holds for any $n \geq N, f \in S \cap(\cup_{s \geq 0}L^2(\Omega,\cF_s,\PP,H))$. Next,  consider the quantity
\begin{equation}\label{eq:avera}
\frac{1}{T}\int_{s}^{s+T}\EE\|P_N \Phi_{s,t}f\|_{H^1}^2\dd t=
\frac{1}{T}\int_{s}^{s+T}\EE\|A^{\frac{1}{2}}P_N  \Phi_{s,t}f\|^2\dd t\;.
\end{equation}
Since \(A^{\frac{1}{2}}P_N \) is a finite rank operator, Proposition~\ref{proposition:sRAGE} implies that \eqref{eq:avera} converges uniformly for \(f \in S\), so that one can find a \(T_0=T_0(N)\) such that 
\begin{equation}
\label{eq:convergence Kernel}
\left|\frac{1}{T}\int_{s}^{s+T}\EE\|P_N \Phi_{s,t}f\|_{H^1}^2\dd t-\mathbb{E} \tr(P_NAK_0 (f \otimes \bar f))\right|\leq B_0
\end{equation}
for any \(T\geq T_0\). Hence, first we fix \(N_0\) large enough so that 
$$
\mathbb{E} \tr(P_{N_0}AK_0 (f \otimes \bar f))\geq 2B_0\:,
$$
for any \(f\in S \cap( \cup_{s \geq 0}L^2(\Omega,\cF_s,\PP,H))\). Then, pick \(T_0\) as above so that \eqref{eq:convergence Kernel} holds for this \(N_0\). It then follows that, for any \(T\geq T_0\), 
\begin{equation}
\frac{1}{T}\int_{s}^{s+T}\EE\|P_{N_0} \Phi_{s,t}f\|_{H^1}^2\dd t \geq B_0 \:,
\end{equation}
and since the expression above is increasing in \(N_0\), the conclusion  also holds for any \(N \geq N_0\).
\end{proof}

\subsection{Enhanced Dissipation and the Proof of Theorem~\ref{Main Abstract Result}}
We now have all the ingredients required for the proof of the main result. Notice that solutions to \eqref{eq:abstrdiff} satisfy the energy equation
\begin{equation}\label{eq:energy}
    \ddt \EE\|f^\nu\|^2+2\nu \EE\|f^\nu\|^2_{H^1}=0 \:.
\end{equation}
\begin{proof}[Proof of Theorem~\ref{Main Abstract Result}]
We begin by proving the harder direction, namely enhanced dissipation for solutions to \eqref{eq:abstrdiff} under the non-existence of invariant $H^1$ subspaces for the $L_k$'s. Thus, having fixed \(\tau,\delta >0\), we need to prove that 
\begin{equation}
\EE\|f^{\nu}(\tau/\nu)\|^2< \delta \:, 
\end{equation}
for all $\nu>0$ sufficiently small and all initial data $f_0\in H$ with $\|f_0\|=1$.

Pick \(M\) large enough so that \(\e^{-\frac{\lambda_M\tau}{4}}<\delta\), and define the set \(K=\{\phi:\EE\|\phi\|^2 \leq 1, \EE\|\phi\|_{H^1}^2 \leq \lambda_M\}\). Note that for \(\phi \in K\), it holds that 
\begin{equation}
\EE(\phi \otimes \overline{\phi})\in \cK\;,\quad\text{with}\quad 
\cK =\{\psi \in \sym:\|\psi\|_{H^1\otimes H^1}^2\leq 4\lambda_M\}\;.
\end{equation}
Indeed, using convexity of the norm we have
\begin{equation}
\|(A^\frac{1}{2}\otimes \Id +\Id \otimes A^\frac{1}{2})\EE(\phi \otimes\overline{\phi})\|\leq \EE\|A^\frac{1}{2}\phi\otimes\overline{\phi}\|+\EE\|\phi \otimes A^\frac{1}{2} \overline{\phi}\|
=2\EE\|A^\frac{1}{2}\phi\|\|\phi\| \leq 2\sqrt{\lambda_M}\:.
\end{equation}
By Lemma~\ref{self-adjointtensor}, \(\cK\) is a compact subset of \(\sym \), and so pick first \(N_1\) so that \(N_1\geq M\), and set \(K_1=\{\phi \in K:\|K_0 \EE(\phi\otimes \overline{\phi})\|\geq \frac{1}{4 N_1}\}\). Next, pick \(N_2\geq N_1\) large enough so that the result in Lemma~\ref{pure point growth} with \(B_0=2\lambda_M\) holds on the set \(K_1\), for some \(T_0\) large enough. Then, apply Proposition~\ref{proposition:sRAGE} to pick \(\tau_1 \geq T_0\) so that for any \(t>\tau_1\), \(\phi \in K\), 
\begin{equation}
\biggl|\frac{1}{t}\int_{s}^{s+t}\EE\|P_{N_1}\Phi_{s,s'}\phi\|^2\dd s'-\sum_{i=1}^{N_1} \langle e_i\otimes \overline{e_i},K_0\EE(\phi \otimes \overline{\phi})\rangle \biggr|<\frac{\lambda_M}{20\lambda_{N_1}}\;.
\end{equation} 
Finally, pick \(\nu_0>0\) so that \(\tau_1<\frac{\tau}{2\nu_0}\), and 
\begin{equation}
\nu_0 \int_0^{\tau_1} B(t)\dd t <\frac{1}{10 \lambda_{N_2}}\:,
\end{equation}
where $B(t)$ is from~\ref{Assumption2}.
Let \(\nu<\nu_0\).
If \(\EE\|f^\nu(t)\|^2_{H^1} \geq \lambda_M \EE\|f^\nu(t)\|^2\) for all \(t\in [0,\tau/\nu]\), then, from the energy equation \eqref{eq:energy},  \(\EE\|f^\nu(\tau/\nu)\|^2 \leq \e^{-2\lambda_M\tau} <\delta\),  and we are done. Otherwise, let \(\tau_0\in [0,\tau/\nu]\) be the first time so that \(\EE\|f^\nu(\tau_0)\|^2_{H^1} < \lambda_M \EE\|f^\nu(\tau_0)\|^2\). We claim that 
\begin{equation}\label{eq:claimexp}
\EE\|f^\nu(\tau_0+\tau_1)\|^2 \leq \e^{-\frac12 \lambda_M \nu \tau_1 }\EE\|f^\nu(\tau_0)\|^2\:.
\end{equation}
Indeed, let $f(t)=\Phi_{\tau_0,t}f^\nu(\tau_0)$ denote the inviscid solution to \eqref{eq:InviscidAbstract}.
Then the difference \(\rho = f^\nu-f\)  satisfies
\begin{equation}
\dd \rho+\sum_{k}iL_k\rho\circ \dd W^k_t =-\nu A f^\nu \dd t\:,
\end{equation}
so that we see 
\begin{equation}
\ddt \|\rho\|^2=2\nu \langle -A f^\nu, \rho\rangle \leq 2\nu(\|f\|_{H^1}\|f^\nu\|_{H^1}-\|f^\nu\|_{H^1}^2) \leq \nu\|f\|^2_{H^1}\:.
\end{equation}
Integrating the above inequality and using~\ref{Assumption2} gives
\begin{equation}
\label{small difference}
\EE\|\rho(t)\|^2  \leq \nu_0\EE\|f^\nu(\tau_0)\|_{H^1}^2 \int_{0}^{t}B(s')\dd s' \leq \frac{1}{10\lambda_{N_2}}\EE\|f^\nu(\tau_0)\|_{H^1}^2 \leq \frac{\lambda_M}{10\lambda_{N_2}}\EE\|f^\nu(\tau_0)\|^2\:.
\end{equation}
Now, we split into two cases: 
\begin{enumerate}[label=(Case \arabic*), ref=(Case \arabic*)]
    \item\label{case1}  \(\|K_0\EE(f^\nu(\tau_0)\otimes \overline{f^\nu(\tau_0)})\| < \frac{1}{4N_1}\EE\|f^\nu(\tau_0)\|^2\), so that 
    \begin{equation}\label{eq:small-lim}
        \biggl|\sum_{i=1}^{N_1} \langle K_0\EE\left(f^\nu(\tau_0)\otimes \overline{f^\nu(\tau_0)}\right),e_i\otimes \overline{e_i}\rangle\biggr|\leq \frac14 \EE\|f^\nu(\tau_0)\|^2;
    \end{equation}
    \item\label{case2} \(\|K_0\EE(f^\nu(\tau_0)\otimes \overline{f^\nu(\tau_0)})\| \geq \frac{1}{4N_1}\EE\|f^\nu(\tau_0)\|^2\).
\end{enumerate}
Let us deal with~\ref{case1} first. Since
\begin{equation}\label{eq:lowhigh}
\|(\Id-P_{N_1})f(t)\|^2=\|f(t)\|^2-\|P_{N_1}f(t)\|^2
\end{equation}
and 
\begin{equation}
\frac{f^\nu(\tau_0)}{(\EE\|f^\nu(\tau_0)\|^2)^\frac{1}{2}} \in K\:,
\end{equation} 
we integrate \eqref{eq:lowhigh}, use that $\lambda_M\leq\lambda_{N_1}$ and \eqref{eq:small-lim}, and obtain
\begin{equation}
\frac{1}{\tau_1}\int_{\tau_0}^{\tau_0+\tau_1}\EE\|(\Id-P_{N_1})f(t)\|^2\dd t \geq \EE\|f^\nu(\tau_0)\|^2 -\left(\frac{\lambda_M}{20\lambda_{N_1}}+\frac{1}{4}\right)\EE\|f^\nu(\tau_0)\|^2 \geq \frac{7}{10}\EE\|f^\nu(\tau_0)\|^2\:.
\end{equation}
Hence, by \eqref{small difference} and $\lambda_M\leq \lambda_{N_2},$ we see that 
\begin{align}
\frac{1}{\tau_1}\int_{\tau_0}^{\tau_0+\tau_1}\EE\|(\Id-P_{N_1})f^\nu(t)\|^2\dd t 
&\geq \frac{1}{\tau_1}\int_{\tau_0}^{\tau_0+\tau_1}\EE\left(\frac{1}{2}\|(\Id-P_{N_1})f(t)\|^2-\|(f(t)-f^\nu(t))\|^2\right)\dd t\notag\\
&\geq \frac{7}{20}\EE\|f^\nu(\tau_0)\|^2-\frac{\lambda_M}{10\lambda_{N_2}}\EE\|f^\nu(\tau_0)\|^2 \geq \frac{1}{4}\EE\|f^\nu(\tau_0)\|^2\:.
\end{align}
It therefore follows that 
\begin{equation}
\int_{\tau_0}^{\tau_0+\tau_1}\EE\|f^\nu(t)\|_{H^1}^2\dd t \geq \frac{\lambda_{N_1}\tau_1}{4}\EE\|f^\nu(\tau_0)\|^2\:.
\end{equation}
Thus, since by the energy equation \eqref{eq:energy} we have the estimate 
\begin{equation}
\EE\|f^\nu(\tau_0+\tau_1)\|^2 \leq \EE\|f^\nu(\tau)\|^2-2\nu \int_{\tau_0}^{\tau_0+\tau_1}\EE\|f^\nu(t)\|_{H^1}^2\dd t \:,
\end{equation}
we get
\begin{equation}
\EE\|f^\nu(\tau_0+\tau_1)\|^2 \leq\EE\|f^\nu(\tau_0)\|^2\left(1-\frac12\nu\lambda_{N_1}\tau_1\right) \leq \e^{- \frac12\nu\lambda_{N_1}\tau_1}\EE\|f^\nu(\tau_0)\|^2 \leq \e^{-\frac{1}{2}\lambda_M\nu \tau_1}\EE\|f^\nu(\tau_0)\|^2\:,
\end{equation}
which implies our initial claim \eqref{eq:claimexp}.

Next, we deal with~\ref{case2}, for which we have
$$
\frac{f^\nu(\tau_0)}{(\EE\|f^\nu(\tau_0)\|^2)^\frac{1}{2}}\in K_1\:.
$$
Therefore, by assumption on $N_2$, the conclusion of Lemma \ref{pure point growth} holds and so we have that 
\begin{equation}
\frac{1}{\tau_1}\int_{\tau_0}^{\tau_0+\tau_1}\EE\|P_{N_2}f(t)\|^2_{H^1}\dd t \geq 2\lambda_M \EE\|f^\nu(\tau_0)\|^2\:.
\end{equation}
Furthermore, by \eqref{small difference}, we see that 
\begin{equation}
\EE\|P_{N_2}(f^\nu(t)-f(t))\|_{H^1}^2 \leq \frac{\lambda_M}{10}\EE\|f^\nu(\tau_0)\|^2\:,
\end{equation}
so that 
\begin{equation}
\frac{1}{\tau_1}\int_{\tau_0}^{\tau_0+\tau_1}\EE\|P_{N_2}f^\nu(t)\|^2_{H^1}\dd t \geq \left(\frac{1}{2}(2\lambda_M)-\frac{\lambda_M}{10}\right)\EE\|f^{\nu}(\tau_0)\|^2 \geq \frac{\lambda_M}{2}\EE\|f^\nu(\tau_0)\|^2\:.
\end{equation}
Thus, we have that 
\begin{equation}
\EE\|f^\nu(\tau_0+\tau_1)\|^2 \leq \EE\|f^\nu(\tau_0)\|^2(1-\nu \lambda_M\tau_1) \leq \e^{-\lambda_M\nu \tau_1}\EE\|f^\nu(\tau_0)\|^2\:,
\end{equation}
which is again \eqref{eq:claimexp}.
We now claim that there exists \(\tau_2 \in [\frac{\tau}{2\nu}, \frac{\tau}{\nu}]\) so that 
\begin{equation}
\EE\|f^\nu(\tau_2)\|^2 \leq \e^{-\frac{1}{4} \lambda_M\nu \tau_2} <\delta\:.
\end{equation}
If \(\tau_0+\tau_1 \geq \frac{\tau}{2\nu}\), then we are done. Otherwise, thanks to the uniformity of $\tau_1$, we may iterate this bound until we find some \(\tau_2 \in [\frac{\tau}{2\nu},\frac{\tau}{\nu}]\) so that 
\begin{equation}
\EE\|f^\nu(\tau_2)\|^2 \leq \e^{-\frac{1}{2}\lambda_M \nu \tau_2}\leq \e^{-\frac{1}{4}\tau \lambda_M} < \delta\:.
\end{equation}
Thus, since \(\tau/\nu\geq \tau_2\), it follows that 
\begin{equation}
\EE\|f^\nu(\tau/\nu)\|^2 \leq \EE\|f^\nu(\tau_2)\|^2 <\delta\:,
\end{equation}
completing the first implication.

For the reverse implication, let \(V\subset H^1\) be a finite dimensional space invariant for all the \(L_k\). Let \(\phi_1,\dots, \phi_m\) be an orthonormal basis for \(V\), and set the vector \(\Phi=(\phi_1,\dots, \phi_m)^T\). Then, for all \(k\), there exists an anti-hermitian matrix \(B_k\), so that \(iL_k \Phi =B_k \Phi\). For a solution \(f^\nu(t)\) to \eqref{eq:abstrdiff}, we then consider the system of equations
\begin{equation}
\label{eq:inner product finite invariant}
\dd\langle f^\nu, \Phi\rangle =-\sum_{k} \langle iL_k f^\nu(t),\Phi\rangle \circ \dd W^k_t+\nu \langle -A f^\nu(t),\Phi\rangle \dd t\:.
\end{equation}
Here we denote by $\langle f^\nu, \Phi\rangle$ the vector which has $j^{th}$ component $\langle f^\nu, \Phi\rangle_j=\langle f^\nu,\phi_j\rangle$.
Using the antisymmetry and the expressions for the actions of \(iL_k\) on \(\Phi\) from the above, we see that \eqref{eq:inner product finite invariant} is equal to 
\begin{equation}
\dd\langle f^\nu, \Phi\rangle =\sum_{k} \overline{B_k}\langle f^\nu(t),\Phi\rangle \circ \dd W^k_t+\nu \langle -A f^\nu(t),\Phi\rangle \dd t\:.
\end{equation}
Applying It\^o's formula, we see that (denoting $v=\langle f^\nu(t),\Phi\rangle$)
\begin{equation}
\label{eq:lower bound computation}
\dd |v|^2 +\sum_{k}2 \Re (v^TB_k\overline{v}) \circ dW^k_t=-2\nu \Re v^T \overline{\langle A^\frac{1}{2}f^\nu(t), A^\frac{1}{2}\Phi\rangle}\:.
\end{equation}
In particular, since the $B_k$ are anti-hermitian, it holds that $\Re v^TB_k \overline{v}=\Re\langle v, \overline{B_k}v\rangle_{\mathbb{C}^m}=0$.
Therefore, we observe the energy balance equality
\begin{equation}
\label{eq:energy dissipation lower bound}
\ddt |v|^2 = -2\nu \Re v^T\overline{\langle A^\frac{1}{2}f^\nu(t),A^\frac{1}{2}\Phi\rangle}\:.
\end{equation}
Furthermore, note the upper bound
\begin{equation}
|v^T\overline{\langle A^\frac{1}{2}f^\nu(t),A^\frac{1}{2}\Phi\rangle}| \leq |v|\|f^\nu(t)\|_{H^1}\|\Phi\|_{H^1} \leq \frac{1}{2}\|\Phi\|_{H^1}^2|v|^2+\frac{1}{2}\|f^\nu(t)\|_{H^1}^2\:.
\end{equation}
Therefore, we see from \eqref{eq:energy dissipation lower bound} that 
\begin{equation}
\frac{\dd }{\dd t}|v(t)|^2 \geq -\nu (\|\Phi\|_{H^1}^2|v(t)|^2+\|f^\nu(t)\|_{H^1}^2)\:,
\end{equation}
which implies the inequality
\begin{equation}
\frac{\dd}{\dd t} (\e^{\|\Phi\|^2_{H^1}\nu t}|v(t)|^2) \geq -\nu \|f^\nu(t)\|^2 \e^{\|\Phi\|^2_{H^1}\nu t}\:.
\end{equation}
Therefore, integrating this inequality yields
\begin{equation}
\label{eq:lower bound v}
|v(t)|^2 \geq \e^{-\|\Phi\|^2_{H^1}\nu t}|v(0)|^2-\nu \left (\int_0^t \e^{-\|\Phi\|^2_{H^1}\nu (t-s)}\|f^\nu(s)\|_{H^1}  ds\right )\:.
\end{equation}
Pick now e.g. \(f^\nu(0)=\phi_1\). By the energy equation \eqref{eq:energy}, we have
\begin{equation}
\nu \int_{0}^\infty \|f^\nu(s)\|_{H^1}^2\dd s \leq \frac{1}{2}\:.
\end{equation}
Since $\e^{-\|\Phi\|^2_{H^1}\nu (t-s)} \leq 1$, for all $s \leq t$, we therefore bound \eqref{eq:lower bound v} by 
\begin{equation}
|v(t)|^2 \geq \e^{-\|\Phi\|^2_{H^1}\nu t}-\frac{1}{2}\:.
\end{equation}
For any $t \leq \frac{1}{\alpha \nu \|\Phi\|_{H^1}^2}$, we thus see that 
\begin{equation}
\|f^\nu(t)\|^2 \|\Phi\|_{H^1}^2 \geq \e^{-\frac{1}{\alpha}}-\frac{1}{2}\:.
\end{equation}
Taking $\alpha$ large enough, we see that for all $t \leq \frac{1}{\nu \alpha \|\Phi\|^2_{H^1}}$ there holds uniformly as $\nu \to 0$ that
\begin{equation}
\|f^\nu(t)\|^2 \geq \frac{1}{4\|\Phi\|^2_{H^1}}\:,
\end{equation}
and thus we indeed do not have enhanced dissipation.
\end{proof}

\subsection{Examples}
The main class of examples we have in mind are drift-diffusion equations of the form \eqref{eq:StochAdvDiff}, in the setting described in Remark~\ref{rmk:concrete}. Hence, we fix a mean-free initial condition $f_0\in L^2(\TT^d)$, as well as a collection \(\sigma_k\) of \(C^{2,\alpha}\) divergence-free vector fields (for some \(\alpha>0\)). As mentioned in Section~\ref{sub:transportnoise}, the corresponding stochastic transport equation 
\begin{equation}
\dd f +\sum_{k=1}^{\infty} \sigma_k \cdot \nabla f  \circ \dd W^k_t=0
\end{equation}
generates a \(C^2\) stochastic flow of diffeomorphisms 
\begin{equation}\label{eq:stoch-char}
\phi_{s,t}(x)=x+\sum_{k=1}^\infty\int_{s}^t \sigma_k(\phi_{s,z}(x))\circ \dd W^k_z\;,
\end{equation}
yielding weak solutions by the formula \(f(t,x)=f_0(\phi_{0,t}^{-1}(x))\) (see \cite{Kunita_1997}).
In particular, it follows readily that for \(f_0 \in C^3\), the above is in fact a strong solution. Further, the conservation of the $L^2$ norm is
a consequence of the divergence free condition. Finally, note that 
\begin{equation}
\EE\|f_0\circ \phi_{0,t}^{-1}\|^2_{\dot{H}^1}\leq \EE\int_{\mathbb{T}^d}|\nabla f_0(\phi_{0,t}^{-1}(x))|^2|D\phi_{0,t}^{-1}(x)|^2\dd x\:.
\end{equation}
Changing variables and conditioning on \(\cF_0\), we get that (since \(\phi_{0,t}\) is independent of \(\cF_0\))
\begin{align}
\EE\int_{\mathbb{T}^d}|\nabla f_0(\phi_{0,t}^{-1}(x))|^2|D\phi_{0,t}^{-1}(x)|^2\dd x&\leq \EE\int_{\mathbb{T}^d}|\nabla f_0(x)|^2\left|(D\phi_{0,t}(x))^{-1}\right|^2\dd x\\
&\leq \EE\|D\phi_{0,t}^{-1}\|_{L^\infty}^2\EE\|f_0\|_{\dot{H}^1}^2\;,
\end{align}
so that our assumptions~\ref{Assumption1}-\ref{Assumption2} are satisfied as long as \(\EE\|D\phi_{0,t}^{-1}\|_{L^\infty}^2\) is a finite function in \(L^1_\loc(0,\infty)\), which follows for instance from \cite[Lem.~3.9]{Gess_Yaroslavtsev}. Hence, such equations fall within our general framework.

\subsubsection{Stochastic Shear Flows}

A particularly simple class of equations on $\TT^2$ to analyze within our framework are those of the form \eqref{eq:StochAdvDiff} where the \(\sigma_j\cdot \nabla \) are shear flows, i.e.\ of the form \(u_j(y)\partial_x\). Indeed, considering the eigenfunction equation and taking Fourier transforms in \(x\), we 
reduce it to the problem
\begin{equation}
(\ell u_j(y)- \lambda)\hat{g}(\ell,y)=0, \qquad \ell\in \ZZ\:.
\end{equation}
If there does not exists a set of non-zero measure on which \(u_j(y)\) is constant, then the only way this can happen is if \(\hat{g}\) is supported at \(\ell=0\), i.e. if it is independent of \(x\). Therefore, it is natural to pose the shear flow SPDE on the space of functions which are mean zero in $x$. Furthermore, if there does exist a non-null set where \(u_j\) is constant, then it must hold that any eigenfunction of $u_j(y) \partial_x$ is a linear combination of the terms of the form $\e^{i\ell x}g_k(y)$, $\supp(g_k) \subset \{u_j =\text{constant}\}$. Since any finite dimensional subspace $V$ of $H=\{\phi \in L^2(\mathbb{T}^2):\int_{\mathbb{T}}\phi(x,y)\dd x=0\}$ which is invariant for $u_j(y) \partial_x$ is spanned by such eigenfunctions, it holds that $P_\ell V$ consists entirely of functions supported on sets of positive Lebesgue measure where $u_j$ is constant. Here, $P_\ell$ is the projection onto the $\ell^{th
}$ Fourier mode in $x$. Therefore, we deduce that any finite dimensional subspace which is invariant for all the \(u_j(y)\partial_x\) is contained in the span of the set 
\begin{equation}
\{\e^{i\ell x}g_\ell(y):\mathrm{supp}(g_\ell)\subset E \}\:,
\end{equation}
where 
$$
E=\bigcap_{j}\bigcup_{z \in \mathbb{R}:\text{Leb}(u_j^{-1}\{z\})>0}u_j^{-1}\{z\}\:.
$$
Since \(H^1\) embeds into the continuous functions in $1d$, we thus can see that stochastic shear flows are dissipation enhancing on \(H=\{\phi \in L^2(\mathbb{T}^2):\int_{\mathbb{T}}\phi(x,y)\dd x=0\}\) if and only if \(E\) does not contain an open interval. We shall say a lot more about the case of shear flows, giving explicit enhanced dissipation rates, in Section~\ref{sec:stochshear}.

\subsubsection{General Hypoelliptic Equations}
More general stochastic-transport equations arise when the vector fields
\(\{\sigma_1, \sigma_2,\dots\}\) satisfy the strong H\"ormander condition. From \cite{Gess_Yaroslavtsev}, it is known that under some further hypoellipticity assumptions on the projective and the two point process associated to \eqref{eq:stoch-char}, the solution to \eqref{eq:StochAdvDiff}
undergoes enhanced dissipation at a time-scale $O(\lvert\log\nu\rvert)$. 
In fact, we show below that hypoellipticity of the generator of the two-point motion away from the diagonal is sufficient to guarantee enhanced dissipation.
\begin{lemma}
Let $H=\{f \in L^2(\mathbb{T}^d): \int_{\mathbb{T}^d}f(x)\dd x=0\}$, and let the $iL_k$ be of the form $\sigma_k\cdot \nabla$ for $\sigma_k \in C^\infty$ divergence free, real vector fields. Then, if the operator $L=-\sum_{k}(L_k \otimes Id-Id\otimes \overline{L_k})^2$ satisfies H\"ormander's bracket condition away from the diagonal $\{x=y\}$, the SPDE is dissipation enhancing on $H$.
\end{lemma}
\begin{proof}
Suppose for a contradiction that this is not the case. By the classical proof of H\"ormander \cite{Hormander_2009}, if \(J g=\lambda g\) for a hypoelliptic operator \(J\), then \(g\) is smooth. In particular, if the $iL_k$ are not dissipation enhancing, there exists a finite dimensional subspace $V \subset H$ that is invariant for all the $iL_k$. Furthermore, the operator $J=\sum_k (iL_k)^2$ is hypoelliptic, and maps $V$ to itself. Hence, $V \subset C^\infty$. Pick an orthonormal basis $\{\phi_j\}_{j=1}^n$ of $V$. By Lemma \ref{lem:KerL}, we note that $\sum_{j=1}^n \phi_j(x)\overline{\phi_j(y)}$ lives in $\Ker(L)$. In fact, since $L$ maps real functions into real functions, writing $\phi_j=a_j+ib_j$, we have that 
\begin{equation}
f(x,y)=\sum_{j=1}^n a_j(x)a_j(y)+b_j(x)b_j(y) \in \Ker(L)\:.
\end{equation}
But now, since $a_j, b_j \in H$, they are mean free, and hence it holds that $\int_{\mathbb{T}^{2d}}f(x,y)\dd x\dd y=0$. Therefore, either $f$ is identically zero, or there exists some $(x_0,y_0) \in \mathbb{T}^{2d}$ so that $f(x_0,y_0)=\min_{x,y \in \mathbb{T}^{2d}}f(x,y)<0$. Note however that $x_0 \neq y_0$, since $f(x,x) \geq 0$ for any $x \in \mathbb{T}^d$, and since $f$ is smooth, it follows by continuity that there exists a connected open set $\mathcal{U} \Subset \mathbb{T}^{2d}\setminus \{x=y\}$, so that $(x_0,y_0) \in \mathcal{U}$. Therefore, since $L$ satisfies H\"ormander's bracket condition on $\mathcal{U}$, by \cite{Battaglia_Biagi_Bonfiglioli_2016} the strong maximum principle holds, and in particular we have that $f \equiv \text{constant}$ on $\mathcal{U}$. Since $\mathcal{U} \Subset \mathbb{T}^{2d}\setminus \{x-y\}$ was arbitrary, $f \equiv \text{constant}$ on $\mathbb{T}^{2d}$, which is a contradiction, completing the proof.
\end{proof}

Our criterion in Theorem~\ref{main result} also gives a non-quantitative enhanced dissipation result for flows that do not fall into the framework of \cite{Gess_Yaroslavtsev}. For instance, we see that the equation
\begin{equation}
\dd f^\nu+\sin(y)\partial_x f^\nu\circ \dd W^1_t+\sin(x)\partial_y f^\nu \circ \dd W^2_t=\nu \Delta f^\nu
\end{equation}
is dissipation enhancing, despite the fact that the generator of the two point motion is not hypoelliptic away from the diagonal. Indeed, the vector fields present in the generator of the two-point motion are (in the notation of \cite{Gess_Yaroslavtsev})
\begin{equation}
\begin{pmatrix}
0\\
\sin(x_1)\\
0\\
\sin(y_1)
\end{pmatrix}
,
\begin{pmatrix}
\sin(x_2)\\
0\\
\sin(y_2)\\
0
\end{pmatrix}\:.
\end{equation}
For instance, at the point \((0,0,\pi,0)\), all the coefficients are zero, so that H\"ormander's bracket condition cannot possibly be satisfied there. However, the eigenfunctions of $\sin(y) \partial_x$ are all those functions depending solely on $y$, whilst the eigenfunctions of $\sin(x)\partial_y$ are all functions depending solely on $x$. The intersection of the eigenspaces of these operators is thus simply the space of constant functions.

\section{Stochastic Shear Flows}\label{sec:stochshear}
A specific class of transport noise we can handle explicitly is that of stochastic shear flows (see \eqref{Eq:StochasticShear}). Since any function depending only on $y$ spans a finite-dimensional invariant subspace of the dynamics, dissipation enhancement cannot hold across the entire space $L^2$. To address this, we restrict \eqref{Eq:StochasticShear} to the Hilbert space
$$
H=\left\{g \in L^2 :\int_{\mathbb{T}}g(x,y)\dd x=0\right\}\:,
$$
which remains invariant under the SPDE.

Taking the Fourier transform of \eqref{Eq:StochasticShear} in the $x$-direction and defining 
\begin{equation}
g_\ell(t,y,y')=\EE(\hat{f_\ell}(t,y)\overline{\hat{f_\ell}(t,y')})\;,
\end{equation}
we obtain the evolution equation
\begin{equation}
\partial_t g_\ell =\nu \Delta g_\ell -2\ell^2\kappa \sum_{j}(u_j(y)-u_j(y'))^2g_\ell\:,
\end{equation}
where $\Delta$  denotes the full Laplacian in the variables $y,y'$. For notational simplicity, we henceforth replace $y'$ with $x$. The right-hand side can then be interpreted as a Schrödinger operator with a non-positive potential. We define
$$
L^\lambda=- \Delta +\lambda^2 \sum_{j}(u_j(y)-u_j(x))^2, \qquad \text{with}\qquad \lambda^2 =\frac{\ell^2 \kappa}{\nu}\:,
$$
and focus on establishing lower bounds for the smallest eigenvalue of $L^\lambda$ in the limit \(\lambda \to \infty\).

In the classical setting of Schrödinger operators with non-negative potentials having finitely many isolated, non-degenerate zeros, this spectral problem has been thoroughly studied (see e.g. \cite{Simon_1983}). However, the present case is more intricate, as the potential may vanish along multiple intersecting curves, complicating direct application of the methods from \cite{Simon_1983}. Despite these challenges, we will establish the following general result.

\begin{theorem}
\label{SchrodingerScaling}
Let \(V=\sum_{j}(u_j(y)-u_j(y'))^2\), and suppose that each $u_j$ has only finitely many isolated critical points. Assume further that for every $x \in \mathbb{T}$, there exists an index $j \in \mathbb{N}$ and an integer $n \leq n_0$ such that $u_j^{(n+1)}(x) \neq 0$. Then there exists a constant $c>0$ such that the smallest eigenvalue of  \(L^\lambda \) is bounded below by \(c \lambda^\frac{2}{n_0+2}\), for all sufficiently large \(\lambda\).
\end{theorem}

Once this result is established, enhanced dissipation follows via a Borel–Cantelli-type argument, as in \cite{BBSP21,Blumenthal_CotiZelati_Gvalani_2023,Gess_Yaroslavtsev}. For completeness, we provide the derivation below, which leads to the following corollary. As before,  $n_0$ denotes the maximal order of overlapping critical points of the functions $u_j$.

\begin{corollary}
\label{corollary:almost sure decay}
Under the assumptions of Theorem~\ref{SchrodingerScaling}, let
 $P_\ell$ be the projection onto the $\ell$-th Fourier mode in $x$, and let $f(t)$ be the solution to \eqref{Eq:StochasticShear} with initial data $f_0$. Then there exist deterministic constants  $\eps_0, \alpha_0>0$, so that for any $\eps \leq \eps_0$, $|\ell|\sqrt{\frac{\kappa}{\nu}}=\alpha \geq \alpha_0$, there exists an almost surely finite random constant
$C_{\nu,\ell,\kappa,\eps}$, depending only on $\nu,\ell, \kappa$ and $\eps$, such that the following enhanced dissipation estimate holds: 
\begin{equation}
\|P_{\ell}f^\nu(t)\|_{L^2_y}\leq C_{\ell,\nu,\kappa,\eps}|\ell|^a \kappa^\frac{a}{2}\nu^{-\beta} \e^{-\eps \nu^\frac{n_0+1}{n_0+2}|\ell|^\frac{2}{n_0+2}\kappa^\frac{1}{n_0+2}t}\|P_\ell f_0\|_{L^2_y}\:,
\end{equation}
where $a, \beta>0$ are constants depending only on $n_0, \eps$. 
Moreover, for every $p>0$ such that $p+1<\eps^{-1}$, there exists a constant $C_{p,\eps}$ independent of $\ell,\nu$ and $\kappa$, such that $\|C_{\ell,\nu,\kappa,\eps}\|_{L^p(\Omega)} \leq C_{p,\eps}$.
\end{corollary}
\begin{proof}
Denote $f_\ell=P_\ell f$, and note that 
\begin{equation}\label{eq:Llambda}
\EE\|f_\ell(t)\|_{L^2_y}^2=\int_{\mathbb{T}}g_\ell(t,y,y)\dd y \leq \left(\int_{\mathbb{T}}|g_\ell(t,y,y)|^2\dd y\right)^\frac{1}{2}\:. 
\end{equation}
Since \(g_\ell(t,y,y)\) is the trace of \(g_\ell(t,y,y')\) on the diagonal \(\{y=y'\}\), by the trace theorem this integral may be bounded above by \(\|g_\ell\|_{H^1(\mathbb{T}^2)}\). It thus suffices to obtain decay on this quantity. By the positivity of \(V\), it further follows that \(\|\sqrt{\nu L^\lambda} h\|^2 \geq \nu \|h\|^2_{\dot{H}^1}\) for any \(h \in H^1\), and by the spectral theorem (see e.g. \cite{K76}), it follows that
\begin{equation*}
\nu\left\| \e^{-\nu L^\lambda t}g_\ell(0) \right\|_{\dot{H}^1} \leq\left\|\sqrt{\nu L^\lambda} \e^{-\nu L^\lambda t} g_\ell(0)\right\|_{L^2}  
\leq (\nu \mu^\lambda)^\frac{1}{2}\e^{-\nu \mu^\lambda t}\|g_\ell(0)\|_{L^2}, \qquad \forall t >\frac{1}{2\nu \mu^\lambda}\:,
\end{equation*}
where \(\mu^\lambda\) is the smallest eigenvalue of \(L^\lambda\). Therefore, using the definition of $\lambda$ in \eqref{eq:Llambda}, as well as the bounds on the smallest eigenvalue given by Theorem~\ref{SchrodingerScaling}, we obtain that 
\begin{equation}
\mathbb{E}\|f_\ell(t)\|_{L^2}^2 \leq C \nu^{-\frac{n_0+4}{2(n_0+2)}}\kappa^\frac{1}{n_0+2}|\ell|^\frac{2}{n_0+2}\e^{-c\nu^\frac{n_0+1}{n_0+2}\kappa^\frac{1}{n_0+2}|\ell|^\frac{2}{n_0+2}t}\mathbb{E}\|f_\ell(0)\|_{L^2}^2\:,
\end{equation}
for some universal constants $C,c>0$. From here, we follow a Borel--Cantelli type argument inspired by \cite{BBSP21,Blumenthal_CotiZelati_Gvalani_2023,Gess_Yaroslavtsev}. However, it will be a bit simpler, since we do not need to obtain a mixing estimate.
Denote by $\{e_k\}$ the standard Fourier basis of $\mathbb{T}$. Let $\zeta>0$, and set
\begin{equation}
N_{k}=\max \{n: \|e_k(n,y)\|_{L^2}\geq \e^{-\zeta n}\}\:,
\end{equation}
where $e_k(t,y)$ denotes the solution to the SPDE
\begin{equation}
\label{shear spde}
\dd e_k(t,y)+\sqrt{2\kappa}\sum_ji\ell u_j(y)e_k(t,y)\circ \dd W^j_t=\nu \partial^2_y e_k(t,y) \dd t\:,
\end{equation}
with initial data $e_k(y)=\e^{iky}$, which is just the stochastic shear flow equation upon taking the Fourier transform in $x$. By Markov's inequality, we have 
\begin{align}
\mathbb{P}(N_{k} >M) 
&\leq \sum_{n\geq M}\mathbb{P}(\|e_k(n,y)\|_{L^2}\geq \e^{-\zeta n})\notag \\
&\leq \sum_{n \geq M}\e^{2\zeta n}\mathbb{E}\|e_k(n,y)\|_{L^2}^2 \leq (\nu^{-1}\mu^\lambda)^\frac{1}{2}(\nu \mu^\lambda -2\zeta)^{-1}C_{\zeta,\nu \mu^\lambda}\e^{(2\zeta-\nu \mu^\lambda)M}\:,
\end{align}
where $\mu^\lambda$ is the smallest eigenvalue of $L^\lambda$, and we have assumed $2\zeta<\nu \mu^\lambda $. Furthermore, the constant $C_{\zeta, \nu\mu^\lambda}$ is uniformly bounded in $\nu, \lambda$. We then have the bounds 
\begin{equation}
\|e_k(y,n)\|_{L^2}\leq \e^{\zeta (N_{k}-n)}.
\end{equation}
Therefore, let now $\phi$ be any $L^2(\mathbb{T})$ function. We may write 
\begin{equation}
\phi(y)=\sum_{k}\phi_k e_k(y)\:,
\end{equation}
where $\phi_k$ is the $k^{\text{th}}$ Fourier mode. We may then estimate 
\begin{equation}
\label{fourier series bound}
\|\phi(n)\|_{L^2} \leq \sum_{k}|\phi_k|\|e_k(n)\|_{L^2} \leq \sum_{k}|\phi_k|\e^{\zeta(N_k-n)}\:.
\end{equation}
Define now the random variable $K=\max\{|k|:\e^{\zeta N_{k}}>|k|\}$.
We then note that 
\begin{equation}
\mathbb{P}(K >M)\leq \sum_{|k|>M}\mathbb{P}(\e^{\zeta N_{k}}>|k|)\leq C_{\zeta,\nu\mu^\lambda}(\nu^{-1}\mu^\lambda)^\frac{1}{2} (\nu \mu^\lambda -2\zeta)^{-1}\sum_{|k|>M}|k|^\frac{2\zeta-\nu\mu^\lambda}{\zeta}\:.
\end{equation}
Assuming further that $3\zeta <\nu \mu^\lambda$, the term on the right hand side is summable, and may be bounded by $C_{\zeta,\nu\mu^\lambda}(\nu^{-1}\mu^\lambda)^\frac{1}{2}(\nu \mu^\lambda -2\zeta)^{-1}|M|^{1+\frac{2\zeta-\nu \mu^\lambda}{\zeta}}$. We therefore conclude that $K$ is almost surely finite, and hence it follows that the random variable
\begin{equation}
D=\max_{|k|\leq K}\e^{\zeta N_{k}}
\end{equation}
is also almost surely finite. We now note that 
\begin{equation}
\e^{\zeta N_{k}} \leq D|k|\:,
\end{equation}
and so, we may bound the expression in \eqref{fourier series bound} by 
\begin{equation}
\|\phi(n,y)\|_{L^2}\leq D\e^{-\zeta n}(\sum_{k}|k||\phi_k|) \leq  CD\e^{-\zeta n}\|\phi\|_{H^\frac{5}{2}}\;.
\end{equation}
We now extend this to an $H^1$ upper bound. Indeed, by \cite[Lemma 7.1]{Bedrossian_Blumenthal_Punshon-Smith_2022}, if $\phi \in H^1$ is mean zero, then for any $\eps>0$, there exists a mean zero function $\phi_\eps$ so that $\phi_\eps \in H^\frac{5}{2}$, $\|\phi_\eps\|_{L^2}\leq C \|\phi\|_{L^2}$, $\|\phi_\eps-\phi\|_{L^2}\leq C\eps \|\phi\|_{H^1}$ and $\|\phi_\eps\|_{H^\frac{5}{2}} \leq C \eps^{-\frac{3}{2}}\|\phi\|_{H^1}$. Certainly, if $\phi$ is not mean zero, this estimate is also true. Therefore, we have 
\begin{equation}
\|\phi(n)\|_{L^2} \leq \|\phi(n)-\phi_\eps(n)\|_{L^2}+\|\phi_\eps(n)\|_{L^2} \leq C(\eps \|\phi\|_{H^1}+D\e^{-\zeta n}\|\phi_\eps\|_{H^\frac{5}{2}})\:,
\end{equation}
where we have used that the transport SPDE is non-expansive in $L^2$ norm. We further bound this expression by 
\begin{equation}
C(\eps+D\e^{-\zeta n} \eps^{-\frac{3}{2}})\|\phi\|_{H^1}\:,
\end{equation}
for possibly a different constant $C$. Finally, we optimize over $\eps$, picking $\eps=\e^{-\frac{2}{5}\zeta n}$, and therefore obtaining 
\begin{equation}
\|\phi(n)\|_{L^2} \leq \hat{D} \e^{-\frac{2}{5}\zeta n}\|\phi\|_{H^1}\:,
\end{equation}
for any $\phi$ with zero mean, where $\hat{D}$ is some new constant depending solely on $D$.
Finally, we may replicate the arguments from section 3.3 of \cite{Gess_Yaroslavtsev} to extend this to all positive times. Indeed, we have the following:
note that the solution to the SPDE \eqref{shear spde} can be written as 
\begin{equation}
\phi(t,y)=\mathbb{E}_B\phi(0,y-\sqrt{2\nu}B_t)\e^{i\sqrt{2\kappa}\ell\sum_{j}\int_0^tu_j(y-\sqrt{2\nu}(B_t-B_s))\dd W^j_s}\:,
\end{equation}
where the expectation is being taken solely with respect to $B$. Therefore, let now $t>0$ and write it as $t=n+t_0$, where $t_0<1$. By our estimate, we have 
\begin{equation}
\|\phi(n,y)\|_{L^2}\leq \hat{D}\e^{-\frac{2}{5}\zeta n}\|\phi(t_0,y)\|_{H^1}\;.
\end{equation}
Furthermore, using Jensen's inequality, we may estimate the $H^1$ norm as 
\begin{equation}
\|\phi(t_0,y)\|_{H^1}\leq \mathbb{E}\|\phi(0,y-\sqrt{2\nu}B_{t_0})\e^{i\sqrt{2\kappa}\ell\sum_{j}\int_0^{t_0}u_j(y-\sqrt{2\nu}(B_{t_0}-B_s))\dd W^j_s}\|_{H^1}\:.
\end{equation}
We therefore just need to bound this final term. Writing the expression for the $H^1$ norm and changing variables $y\mapsto y-\sqrt{2\nu}B_{t_0}$, we need to bound
\begin{equation}
\|\partial_y \psi(0)\|_{L^2}+\biggl(\int_{\mathbb{T}}|\psi(0,y)|^2|\sum_{j}\sqrt{2\kappa}\ell\int_0^{t_0}u'_j(y+\sqrt{2\nu}B_s) \dd W_s^j|^2\dd y\biggr)^\frac{1}{2}\;.
\end{equation}
It remains to bound the stochastic integral term. Denoting $I(t,x)=\sum_{j}\int_0^{t}u'_j(x+\sqrt{2\nu}B_s) \dd W_s^j$, we have by the Itô isometry (say $s<t$) that
\begin{align}
|I(t,x)-I(s,y)|^2 
&\leq 2\sum_j\int_0^{s}(u'_j(x+\sqrt{2\nu}B_z)-u'_j(y+\sqrt{2\nu}B_z))^2\dd z+\int_s^t(u'_j(x+\sqrt{2\nu}B_z))^2\dd z \notag\\
&\leq 2\sum_j\|u_j\|_{C^2}(|x-y|^2+|t-s|)\:,
\end{align}
uniformly for $s<t\leq 1$. By Kolmogorov's continuity criterion (see e.g. \cite{Kunita_1997}), this implies that there exists a square integrable constant $K$ so that 
\begin{equation}
|I(t,x)-I(s,y)|\leq K(|x-y|^\frac{1}{2}+|t-s|^\frac{1}{4})
\end{equation}
almost surely. In fact, one can show that this constant has moments of all orders as a consequence of the BDG
inequality and the standard Kolmogorov continuity criterion, see e.g.\ \cite{Kunita_1997,Gess_Yaroslavtsev}. Finally, picking any reference point $x \in \mathbb{T}$, this shows that 
\begin{equation}
\sup_{t \in [0,1]}\sup_{y \in \mathbb{T}}|I(t,y)| \leq K(|\mathbb{T}|^\frac{1}{2}+1+|I(0,x)|)=K(|\mathbb{T}|^\frac{1}{2}+1)\:,
\end{equation}
which is almost surely finite, and has finite moments of any order, independent of $\nu$. With this in mind, we see that we can pick a constant $\Tilde{C}$ with moments of all orders that are independent of $\nu$ so that 
\begin{equation}
\|\phi(t_0,y)\|_{H^1}\leq (1+\Tilde{C}|\ell|\kappa^\frac{1}{2})\|\phi(0)\|_{H^1}\:,
\end{equation}
for any $t_0 \in [0,1]$. Therefore, we may conclude that 
\begin{equation}
\|\phi(t)\|_{L^2}\leq (1+|\ell|\kappa^\frac{1}{2}\Tilde{C})\hat{D}\e^{-\frac{5}{2}\zeta t}\|\phi(0)\|_{H^1}\;.
\end{equation}
We now set 
\begin{equation}
\tau=\inf \Bigl\{t>0 : \|\nabla \phi(t)\|_{L^2}^2\leq \|\phi(0)\|^2_{L^2}\frac{1}{2}\nu^{-1}\Bigr\}\:.
\end{equation}
By the energy estimate 
\begin{equation}
\|\phi(t)\|_{L^2}^2 =\|\phi(0)\|_{L^2}^2-2\nu\int_0^t\|\nabla \phi(s)\|_{L^2}^2\dd s\:,
\end{equation}
it holds that $\tau \leq 1$, and by the arguments in \cite{Gess_Yaroslavtsev} section 4, this is a stopping time. Therefore, we can run the equation up to time $\tau$, and start the estimate from there, to finally conclude
\begin{equation}
\|\phi(t)\|_{L^2}\leq \hat{D}(1+\Tilde{C}\kappa^\frac{1}{2}|\ell|)\nu^{-\frac{1}{2}}\e^{-\frac{2}{5} \nu \mu^\lambda t}\|\phi(0)\|_{L^2}\:.
\end{equation}
Next, we need to estimate the moments of the constants $\Tilde{C}$, $\hat{D}$. We have already seen that the moments of $\Tilde{C}$ are bounded independently of $\nu, \kappa, |\ell|$. Therefore, it remains to check the moments of $\hat{D}$. We have 
\begin{align*}
\mathbb{E}\hat{D}^p &\leq \mathbb{E}\sum_{K_0=1}^\infty\mathds{1}_{K=K_0}\max_{|k|\leq K_0} \e^{\zeta pN_k} \\&\leq C \sum_{K_0=1}^\infty ((\nu^{-1} \mu^\lambda)^\frac{1}{2}(\nu \mu^\lambda -2\zeta)^{-1}K_0^{1+\frac{2\zeta-\nu \mu^\lambda}{\zeta}})^\frac{1}{2}(\mathbb{E}(\max_{|k|\leq K_0} \e^{\zeta pN_k})^2)^\frac{1}{2}\notag\\
&\leq \sum_{K_0=1}^\infty ((\nu^{-1} \mu^\lambda)^\frac{1}{2}(\nu \mu^\lambda -2\zeta)^{-1}K_0^{1+\frac{2\zeta-\nu \mu^\lambda}{\zeta}})^\frac{1}{2} \sum_{|k|\leq K_0}(\mathbb{E}\e^{2\zeta pN_k})^\frac{1}{2}\:.
\end{align*}
It remains to note that 
\begin{equation}
\mathbb{E}\e^{2\zeta pN_k} \leq \sum_{M\geq 0}\e^{2p\zeta(M+1)}\mathbb{P}(N_k\geq M) \leq C_{p,\zeta,\nu\mu^\lambda} \sum_{M\geq 0}(\nu^{-1}\mu^\lambda)^\frac{1}{2}(\nu \mu^\lambda -2\zeta)^{-1}\e^{M(2p\zeta+2\zeta-\nu\mu^\lambda)}\:.
\end{equation}
As long as $p+1<\frac{\nu \mu^\lambda}{\zeta}$, this is finite and may be bounded above by 
\begin{equation*}
C_{p,\zeta,\nu\mu^\lambda}(\nu^{-1} \mu^\lambda)^\frac{1}{2}(\nu \mu^\lambda -2\zeta)^{-1}\frac{1}{1-\e^{2p\zeta+2\zeta-\nu\mu^\lambda}} \leq C_{p,\zeta,\nu\mu^\lambda}(\nu^{-1}\mu^\lambda)^\frac{1}{2}(\nu \mu^\lambda -2\zeta)^{-1}(2p\zeta+2\zeta-\nu\mu^\lambda)^{-1}\:,
\end{equation*}
where we have used the inequality $\frac{1}{1-\e^{-x}}\leq 2x^{-1}$ for $x$ small. Here, the constant $C_{p,\zeta,\nu \mu^\lambda}$ is bounded independently of $\nu,\mu^\lambda$. Hence, if we take $\zeta =\eps \nu \mu^\lambda$, and observe the scaling of $\mu^\lambda$ from Theorem~\ref{SchrodingerScaling}, we deduce the claimed the result upon taking $C_{\ell,\nu,\kappa,\eps}=\frac{\Tilde{C}\hat{D}}{\left ((\nu^{-1}\mu^\lambda)^\frac{1}{2}(\nu \mu^\lambda -2\zeta)^{-1}(2p\zeta+2\zeta-\nu\mu^\lambda)^{-\frac{1}{2}}\right )^\frac{1}{p}}$\:.
\end{proof}
\begin{remark}
If one is only interested in enhanced dissipation in expectation, the beginning of the proof gives an elementary argument, which yields slightly sharper decay rates in expectation than the ones one would obtain by taking expectations in the almost sure bounds of Corollary~\ref{corollary:almost sure decay}. This is summarized in the following corollary:
\end{remark}
\begin{corollary}
Under the assumptions of  Theorem~\ref{SchrodingerScaling}, there exists a deterministic constant $\alpha_0$ so that for all $|\ell|\sqrt{\frac{\kappa}{\nu}}\geq \alpha_0$, it holds that 
\begin{equation}
\EE\|P_{\ell}f^\nu(t)\|^2_{L^2_y} \leq C \nu^{-\frac{n_0+4}{2(n_0+2)}}\kappa^\frac{1}{n_0+2}|\ell|^\frac{2}{n_0+2}\exp{(-c \nu^\frac{n_0+1}{n_0+2}\kappa^\frac{1}{n_0+2} |\ell|^\frac{2}{n_0+2} t)}\EE\|P_\ell f(0)\|^2_{L^2_y}\:,
\end{equation}
for every $t\geq0$, and some constants $C,c>0$ independent of $|\ell|,\kappa, \nu$.
\end{corollary}

\subsection{Main Ideas of the Proof}
The proof of Theorem~\ref{SchrodingerScaling} is an adaptation of the spectral analysis techniques developed in \cite{Simon_1983}. The central insight is that the asymptotic scaling of the spectrum of a Schrödinger operator with a non-negative potential is primarily governed by the local geometry of the potential’s zero set. Specifically, recall that the smallest eigenvalue of the operator $-\Delta+\lambda^2 V$ can be characterized variationally as the minimum of
\begin{equation}
\int_{\mathbb{T}^2}|\nabla f|^2+\lambda^2 V|f|^2 \dd x\:,
\end{equation}
taken over all functions $f\in H^1(\mathbb{T}^2)$ with  $
\|f\|_{L^2}=1$. As $\lambda \to \infty$, it is natural to expect that minimisers will concentrate near the zero set of $V$, where the potential energy is negligible. Consequently, the structure of the potential near its zero set plays a decisive role in determining the spectral scaling.

To make this intuition rigorous, we employ the Ismagilov–Morgan–Simon (IMS) localization formula (see \cite{Simon_1983}), which allows one to partition the domain and localize the operator to regions where the behavior of the potential is better understood. The hope is then to approximate $V$ near its zeros by a simpler potential $\Tilde{V}$, for which the spectral behavior of $-\Delta +\lambda^2 \Tilde{V}$ is explicitly known. For instance, when all zeros of $V$ are non-degenerate critical points—i.e., $V$ locally resembles $|x|^2$, the model operator $-\Delta +\lambda^2 |x|^2$ is unitarily equivalent to $\lambda(-\Delta + |x|^2)$, and thus its spectrum scales linearly in $\lambda$.

In our setting, the zero set of $V=\sum_{j}(u_j(y)-u_j(y'))^2$ is more intricate than a collection of isolated non-degenerate minima. However, due to the special structure of the functions $u_j$, the types of degeneracies that arise near the zero set are relatively limited. Specifically, the zeros of $V$ are confined to a finite union of smooth, non-intersecting curves along which the Hessian of $V$ is non-vanishing, and a finite number of isolated points where the potential can be locally approximated (in a sense we will make precise) by expressions of the form $(x^{n_1}-y^{n_2})^2$. By analyzing the spectral scaling in neighborhoods of these regions, we will derive the scaling behavior of the full operator in the $\lambda \to \infty$ limit.

\subsection{Proof of Theorem~\ref{SchrodingerScaling}}
We begin by stating the IMS localisation formula, which will be of essential importance to our proof.
\begin{lemma}[\cite{Simon_1983}]
\label{IMS}
Let \(\{J_1^2,\dots J_n^2\}\) be a smooth partition of unity, and let \(H=-\Delta+V_0\) be a Schrödinger operator defined by the form sum of \(-\Delta\) and \(V_0\), see \cite{Reed_Simon_1972}, so that the form domain of \(H\) is precisely \(\cD(-\Delta)\cap \cD(V_0^+)\). Then, it holds that
\begin{equation}
H=\sum_{j=1}^nJ_i H J_i -\sum_{j=1}^n |\nabla J_i|^2\:.
\end{equation}
\end{lemma}
\begin{remark}
The assumption on the domains will certainly hold in our case, since the potential is a bounded operator.
\end{remark}
As mentioned before, this result essentially tells us that we can reduce the global problem of finding bounds on eigenvalues of a Schrödinger operator to a series of local problems. In particular, we shall now seek to understand the behaviour of the local problems that may occur for our operator. We begin by studying the spectrum of 
\begin{equation}
-\Delta+\lambda^2 (x^n-y^m)^2
\end{equation}
on $L^2(\mathbb{R}^2)$. To do so, we first note the following preparatory lemma.
\begin{lemma}
\label{lemma:schrodingerScaling1}
Consider the Schr\"odinger operator $H^\mu=-\partial_x^2+\mu^2 (x^m \pm 1)^2$, i.e. the potential is given either by $(x^m-1)^2$ or $(x^m+1)^2$. Let $\mu >0$. Then, there exists a constant $c>0$, so that for all $\mu>0$, the spectrum of $H^\mu$ is contained in $(c\mu,\infty)$.
\end{lemma}
\begin{proof}
We begin by noting that for $\mu \to \infty$, the result follows by \cite{Simon_1983}, since the function $(x^m \pm 1)^2$ only has isolated, non-degenerate zeros. Hence, it remains to deal with finite $\mu$. However, note further that since $(x^m \pm 1)^2 \geq 0$, the infimum of the spectrum of $H^\mu$ is a non-decreasing function in $\mu$. If we can further prove that for all $\mu$ sufficiently close to $0$, it holds that $\inf \sigma(H^\mu) \geq c_0 \mu$, then $\inf \sigma(H^\mu)$ is a non-decreasing function in $\mu$ so that for all $\mu$ small enough, and for all $\mu$ large enough, there exists a constant $c_0>0$ so that $\inf \sigma(H^\mu) \geq c_0 \mu$. This implies the claimed result. Therefore, it remains to deal with the bound as $\mu \to 0$.
We begin by noting that for any $m, a$ there exists a constant $b>0$ so that 
\begin{equation}
(x^m \pm 1)^2 \geq -1+bx^2\:,
\end{equation}
for all $x \in \mathbb{R}$. Therefore, in the sense of forms, it holds that 
\begin{equation}
H^\mu \geq -\partial_x^2 +\mu^2(bx^2-1)\:.
\end{equation}
Certainly, it holds that $\sigma(-\partial_x^2 +\mu^2(bx^2-1))=\sigma(-\partial_x^2 +b\mu^2x^2)-\mu^2$. Furthermore, note that the unitary map $Uf(x)=\mu^{\frac{1}{4}}f(\mu^\frac{1}{2}x)$ transforms the operator into
\begin{equation}
U^{-1}(-\partial_x^2+b\mu^2x^2)U=\mu (-\partial_x^2+bx^2)\:.
\end{equation}
Hence, it holds that $\sigma(-\partial_x^2+b\mu^2x^2)=\mu \sigma(-\partial_x^2+bx^2) \subset [c\mu, \infty)$, for some $c>0$. Therefore, for $\mu$ small enough it holds that 
\begin{equation}
\sigma(-\partial_x^2 +\mu^2(bx^2-1))=\sigma(-\partial_x^2 +b\mu^2x^2)-\mu^2 \subset [c\mu-\mu^2, \infty) \subset [\frac{c}{2}\mu, \infty)\:.
\end{equation}
This completes the proof.
\end{proof}
With this result in hand, we now have all the tools needed to understand the scaling near a degeneracy of the form $(x^{m}-y^{n})^2$.
\begin{lemma}
\label{lemma:Schrodinger scaling 2}
Consider the Schr\"odinger operator $H=-\Delta+\lambda^2(x^m\pm y^n)^2$ on $\mathbb{R}^2$, for $m,n \in \mathbb{N}$, where the $\pm$ is understood as in Lemma \ref{lemma:schrodingerScaling1}. Then, there exists a constant $c_0>0$ so that the spectrum of $H$ is contained in $[c_0\lambda^\frac{2n}{2n+nm-m}, \infty)$, for all $\lambda \geq 0$.
\end{lemma}
\begin{proof}
Pick any $f \in C_c^\infty(\mathbb{R}^2)$. Then, for almost every $x>0$, we have 
\begin{equation}
\int_{\mathbb{R}}|\partial_y f|^2+\lambda^2(x^m
\pm y^n)^2|f|^2\dd y=\int_{\mathbb{R}}|\partial_y f|^2+\lambda^2x^{2m}\left(1 \pm \left(\frac{y}{x^{\frac{m}{n}}}\right)^n\right)^2|f|^2\dd y\:.
\end{equation}
Change variables $v =\frac{y}{x^\frac{m}{n}}$ to see this is lower bounded by 
\begin{align}
&x^{\frac{m}{n}}\int_{\mathbb{R}}|\partial_v f(v,x)|^2 x^{-\frac{2m}{n}}+ x^{2m}\lambda^2\left(1 \pm v^n\right)^2|f(v)|^2 \dd v\notag\\
&\qquad \quad =x^{-\frac{m}{n}}\int_{\mathbb{R}}|\partial_v f(v,x)|^2 +x^{2m+\frac{2m}{n}}\lambda^2\left(1 \pm v^n\right)^2|f(v)|^2 \dd v\:.
\end{align}
Applying Lemma \ref{lemma:schrodingerScaling1} with $\mu=\lambda x^{m+\frac{m}{n}}$, there holds that 
\begin{align}
x^{-\frac{m}{n}}\int_{\mathbb{R}}|\partial_v f(v,x)|^2 +x^{2m+\frac{2m}{n}}\lambda^2\left(1 \pm v^n\right)^2|f(v,x)|^2 \dd v 
&\geq c_0\lambda x^{m} \int_{\mathbb{R}}|f(v,x)|^2 \dd v\notag\\
&=c_0\lambda x^{m-\frac{m}{n}}\int_{\mathbb{R}}|f(y,x)|^2 \dd y\:.
\end{align}
Similarly, for $x<0$, we can write $(x^m \pm y^n)^2= x^{2m}((-1)^m \pm (\frac{y}{(-x)^\frac{m}{n}})^n)^2$ and the same bound holds for this expression. Hence, we have shown that there exists $c_0>0$ so that for any $f \in C_c^\infty(\mathbb{R}^2)$, and for almost every $x$ there holds that
\begin{equation}
\int_{\mathbb{R}}|\partial_y f|^2+\lambda^2(x^m \pm y^n)^2|f|^2\dd y \geq c_0\lambda |x|^{m-\frac{m}{n}}\int_{\mathbb{R}}|f|^2 \dd y\:.
\end{equation}
Therefore, we have the form bound 
\begin{equation}
\int_{\mathbb{R}^2}|\partial_x f|^2 +|\partial_y f|^2+\lambda^2(x^m \pm y^n)^2|f|^2 \dd x \dd y \geq \int_{\mathbb{R}^2}|\partial_x^2f|+c_0 \lambda|x|^{m-\frac{m}{n}}|f|^2 \dd x \dd y\:.
\end{equation}
This lower bound corresponds to the form associated to the Schr\"odinger operator $(-\partial_x^2+c_0\lambda|x|^{m-\frac{m}{n}})$. This is of the form $-\partial_x^2+c_0\alpha^2|x|^{2\beta}$. Under conjugation by the unitary transform $Vf(x)=\alpha^{\frac{1}{2(\beta+1)}}f(x\alpha^{\frac{1}{\beta+1}})$ this operator is equivalent to $\alpha^{\frac{2}{\beta+1}}(-\partial_x^2+c_0|x|^{2\beta})$. 
Hence, we observe that 
\begin{align}
\int_{\mathbb{R}^2}|\partial_x f|^2+c_0 \lambda|x|^{m-\frac{m}{n}}|f|^2 \dd x \dd y\geq c\lambda^{\frac{2n}{mn+2n-m}}\int_{\mathbb{R}^2}|f|^2 \dd x \dd y\:,
\end{align}
which is the claimed bound.
\end{proof}
Let us remark at this point that, as long as $n,m \leq n_0+1$, it holds that $\frac{2n}{2n+mn-m} \geq \frac{2}{n_0+2}$. Indeed, rearranging this inequality, it is equivalent to show that $n(n_0+2) \geq 2n+mn-m$, i.e. that $n (n_0 -m)\geq -m$. If $n_0 \geq m$ this is trivially satisfied, and if $m \in [n_0,n_0+1]$, this becomes $n \leq \frac{m}{m-n_0}=1+\frac{n_0}{m-n_0}$. But $n \geq n_0+1$, and since $m-n_0 \leq 1 $, we have the chain of inequalities $n \leq 1+n_0 \leq 1+\frac{n_0}{m-n_0}$, and so indeed $\frac{2n}{2n+mn-m} \geq \frac{2}{n_0+2}$, for all natural numbers $m,n \leq n_0+1$.\\

Finally, we show that locally our problem will always look like the above toy scenario.
\begin{lemma}
\label{lemma:LocallyPolynomial}
Let $f \in C^{m+1}(\mathbb{R})$ have a critical point of order $m-1$ at zero, i.e. locally $f$ may be written $f(x)=f(0)+cx^{m}+\mathcal{O}(x^{m+1})$ for $c \neq 0$. Then, there exist open intervals $\Tilde{I}, I \ni 0$ and a diffeomorphism $\psi: I \to \Tilde{I}$ so that $f(\psi^{-1}(x))=f(0)+\mathrm{sign}(c) x^m$.
\end{lemma}
\begin{proof}
Without loss of generality we let $f(0)=0$. Then consider the map ${\mathrm{sign}}(x)|f(x)|^\frac{1}{m}$. We claim that for $\phi^{-1}(x)$ near zero, we have $f(\phi^{-1}(x))=\mathrm{sign}(c)x^m$ if an only if \\
${\mathrm{sign}}(\phi^{-1}(x))|f(\phi^{-1}(x))|^\frac{1}{m}=x$. Indeed, if $m$ is even, taking both sides of the equation to the power $m$ yields that $|f(\phi^{-1}(x))|=x^m$. Furthermore, since $m$ is even, near zero it holds that ${\mathrm{sign}}(f(x))={\mathrm{sign}}(c)$. Thus, it must hold near zero that $f(\phi^{-1}(x))=\mathrm{sign}(c)x^m$. Similarly, if $m$ is odd, then taking the equality ${\mathrm{sign}}(\phi^{-1}(x))|f(\phi^{-1}(x))|^\frac{1}{m}=x$ to the power $m$ yields ${\mathrm{sign}}(\phi^{-1}(x))|f(\phi^{-1}(x))|=x^m$. Keeping in mind that $m$ is odd, for $x$ near zero it holds that ${\mathrm{sign}}(f(x))={\mathrm{sign}}(x){\mathrm{sign}}(c)$. Thus, it is clear that ${\mathrm{sign}}(\phi^{-1}(x))|f(\phi^{-1}(x))|={\mathrm{sign}}(c)f(\phi^{-1}(x))$. Thus, we in fact see that $f(\phi^{-1}(x))=\mathrm{sign}(c) x^m$, as claimed. Therefore, consider now the function $g(x,y)={\mathrm{sign}}(x)|f(x)|^\frac{1}{m}-y$. Certainly $g(0,0)=0$. Furthermore, it holds that 
\begin{equation}
\partial_xg(x,y)=\mathrm{sign}(x)\frac{1}{m}\frac{{\mathrm{sign}}(f(x))f'(x)}{|f(x)|^{1-\frac{1}{m}}}
\end{equation} 
away from zero. In particular, note that 
\begin{align}
\lim_{x \to 0}\partial_x g(x,y)=\lim_{x \to 0}\frac{1}{m}\frac{{\mathrm{sign}}(x)^m {\mathrm{sign}}(c) (mcx^{m-1}+\mathcal{O}(x^{m}))}{|cx^m+\mathcal{O}(x^{m+1})|^{1-\frac{1}{m}}}=\frac{|c|}{|c|^{1-\frac{1}{m}}}=|c|^{\frac{1}{m}}\:.
\end{align}
Furthermore, it is in fact differentiable at zero, and its derivative is given by $|c|^{\frac{1}{m}}$. Indeed, we compute
$$
\lim_{x \to 0}x^{-1}[{\mathrm{sign}}(x)|cx^m+\mathcal{O}(x^{m+1})|^{\frac{1}{m}}-x|c|^{\frac{1}{m}}]=\lim_{x \to 0}x^{-1}[{\mathrm{sign}}(x)|c|^{\frac{1}{m}}|x|(1+\mathcal{O}(x))-x|c|^\frac{1}{m}]=0\:.
$$
Therefore, by the implicit function theorem, there exist open intervals $I, \Tilde{I}$ containing zero, and a unique differentiable function $\chi:\Tilde{I} \to I$ so that $g(\chi(x),x)=0$. In fact, differentiating this expression in $x$, we see that it holds 
\begin{equation}
\partial_xg(\chi(x),x)\partial_x \chi-1=0\:.
\end{equation}
Since we know that $\partial_x\chi(x,y)$ is non-vanishing and continuous in an interval around zero, it is clear that 
\begin{equation}
\partial_x \chi(x)=\frac{1}{\partial_x g(\chi(x),x)}\:.
\end{equation}
The right-hand side of this equation is a Lipschitz function, so we conclude that $\chi(x)$ is in fact $C^1$, and also that $\partial_x \chi(x) \neq 0$ for all $x \in \Tilde{I}$, after possibly shrinking $\Tilde{I}$. Therefore, by the inverse function theorem, $\chi$ defines a diffeomorphism $\chi: \Tilde{I} \to I$. Setting $\psi=\chi^{-1}$, we complete the proof. 
\end{proof}

We now state the following lemma, which describes the local structure of the zero set of a potential of the form $V(x,y)=\sum_{j}(u_j(x)-u_j(y))^2$.
\begin{lemma}[Structure of the zero set]
\label{lemma:structure of zero set}
Let $\{u_j(x)\}_{j}$ be a countable set of $C^{\infty}$ functions on the torus, with a finite number of maximal overlapping critical point of order $n_0$, (i.e. for all $x \in \mathbb{T}$, there exists $j \in \mathbb{N}$ and $n \leq n_0$ so that $u_j^{(n+1)}(x) \neq 0$). Let $V=\sum_{j}(u_j(x)-u_j(y))^2$. Then, for any $\mathbf{x}=(x_0,y_0) \in V^{-1}(\{0\})$, one of the following holds true:
\begin{enumerate}[label=(C\arabic*), ref=(C\arabic*)]
    \item \label{C1} $\mathbf{x}$ is an isolated point of $V^{-1}(\{0\})$, and $V(x,y) \geq c ((x-x_0)^{2n_0+2}+(y-y_0)^{2n_0+2})$ in a neighbourhood of $\mathbf{x}$.
    \item \label{C2} There exists $j \in \mathbb{N}$, and a $C^2$ diffeomorphism $\Phi_{\mathbf{x}}:B_{\mathbf{x}} \to \mathcal{U}_{\mathbf{x}}$ so that setting $\Phi_{\mathbf{x}}(x,y)=(\zeta,\eta)$, $V_j(x,y)=(u_j(x)-u_j(y))^2$, it holds that $\{V_j=0\}\cap B_{\mathbf{x}} =\Phi^{-1}_{\mathbf{x}}\{(\zeta, 0)\}$. Furthermore, it holds that $\partial_\eta^2 V_j(\Phi_{\mathbf{x}}^{-1}(\zeta,0)) >0$.
    \item \label{C3} There exists a diffeomorphism $\Psi_{\mathbf{x}}:B_{\mathbf{x}}\to \mathcal{U}_{\mathbf{x}}$ so that $V(\Psi^{-1}_{\mathbf{x}}(x,y)) \geq ((x-x_0)^{n_1}\pm (y-y_0)^{n_2})^2$, for some $n_1,n_2 \leq n_0+1$, and where the $\pm$ is understood as in Lemma \ref{lemma:schrodingerScaling1}.
\end{enumerate}
\end{lemma}
\begin{proof}
We begin by studying the degenerate set 
\begin{equation}S=\{(x,y):u_j(x)=u_j(y), \ u'_j(x)=u'_j(y)=0 \ \forall j\}\:.
\end{equation}
Let $\mathbf{x}=(x_0,y_0) \in S$. There are now two options. Either, there exists $j \in \mathbb{N}$ so that $u_j(x) \sim c_1(x-x_0)^{n_1}$, $u_j(y) \sim c_2(y-y_0)^{n_2}$, for $n_1, n_2 \leq n_0+1$, or for all $j \in \mathbb{N}$, at least one of $n_1, n_2$ as above is strictly larger than $n_0+1$. In the first case, Lemma~\ref{lemma:LocallyPolynomial} implies the existence of $C^1$ diffeomorphisms $\psi_{x_0}, \psi_{y_0}$, so that $u_j(\psi_{x_0}^{-1}(x))-u_j(\psi_{y_0}^{-1}(y))=(x-x_0)^{n_1} \pm (y-y_0)^{n_2}$. In fact, the maps $\phi_{x_0},\phi_{y_0}$ induce a diffeomorphism via $\Psi_\mathbf{x}=(\psi_{x_0},\psi_{y_0})$, and so we are in case \ref{C3} of Lemma~\ref{lemma:structure of zero set}. Suppose next that for each $j$, at least one of $n_1, n_2$ is strictly larger than $n_0+1$. In this case, we claim that the zero $\mathbf{x}=(x_0,y_0)$ of $V(x,y)$ must be isolated. Indeed, since the maximal order of overlapping zeros of the $u_j$ is $n_0$, there must exist $j_1,j_2$, so that 
\begin{equation*}
u_{j_1}(x) \sim c_1(x-x_0)^{n_1}, \quad u_{j_1}(y)\sim c_2(y-y_0)^{n_2}, \quad u_{j_2}(x) \sim c_3(x-x_i)^{n_3}, \quad u_{j_2}(y)\sim c_4(y-y_i)^{n_4}\:,
\end{equation*}
where $n_1, n_4 \leq n_0+1$, $n_2, n_3 >n_0+1$. Then, there holds the lower bound for $(x,y)$ near $\mathbf{x}$ (changing coordinates so $\mathbf{x}=0$)
\begin{align*}
& (u_{j_1}(x)-u_{j_1}(y))^2+(u_{j_2}(x)-u_{j_2}(y))^2\\
&=(c_1x^{n_1}-c_2y^{n_2}+\mathcal{O}(x^{n_1+1})+\mathcal{O}(y^{n_2+1}))^2+(c_3x^{n_3}-c_4y^{n_4}+\mathcal{O}(x^{n_3+1})+\mathcal{O}(y^{n_4+1}))^2\\
&=c_1^2 x^{2n_1}+c_4^2 y^{2n_4}+\mathcal{O}(x^{2n_1+1})+\mathcal{O}(y^{2n_4+1})+\mathcal{O}(x^{n_1}y^{n_2})+\mathcal{O}(x^{n_3}y^{n_4})\:.
\end{align*}
Noting that $x^{n_1}y^{n_2} \leq \delta x^{2n_1}+C(\delta) y^{2n_2}$, $n_2 > n_4$, and that a similar bound works for $x^{n_3}y^{n_4}$, we actually obtain a lower bound of the form 
\begin{equation}
(u_{j_1}(x)-u_{j_1}(y))^2+(u_{j_2}(x)-u_{j_2}(y))^2 \geq \frac{c_1^2}{2}(x-x_0)^{2n_1}+\frac{c^2_4}{2}(y-y_0)^{2n_4}>0\:,
\end{equation}
for all $x,y$ in a punctured neighbourhood of $\mathbf{x}$. Hence, the zero $\mathbf{x}$ is indeed isolated for $V$, and since $n_1, n_4 \leq n_0+1$, the lower bound of \ref{C1} follows readily. Finally, suppose $\mathbf{x} \notin S$. Then, there exists $j \in \mathbb{N}$ so that say $u'(y_0) \neq 0$. By the inverse function theorem, the map $\Phi_{\mathbf{x}}(x,y)=(x,u_j(x)-u_j(y))$ is a diffeomorphism in a neighbourhood of $\mathbf{x}$ so that $V_j^{-1}(\{0\}) =\Phi_{\mathbf{x}}^{-1}\{(\zeta,0)\}$ close to $\mathbf{x}$. Furthermore, we compute  
\begin{align}
&\partial^2_\eta V_j(\Phi_{\mathbf{x}}^{-1}(\zeta,0))=\partial_x^2 V_j(\Phi_{\mathbf{x}}^{-1}(\zeta,0))(\partial_\eta x)^2\\
+&2\partial_x \partial_y V_j(\Phi_{\mathbf{x}}^{-1}(\zeta,0))\partial_\eta x\partial_\eta y+\partial_y^2 V_j(\Phi_{\mathbf{x}}^{-1}(\zeta,0))(\partial_\eta y)^2\:,
\end{align}
where we have set $(x,y)=\Phi^{-1}_{\mathbf{x}}(\zeta,\eta)$, and used $\partial_x V_j(\Phi_{\mathbf{x}}^{-1}(\zeta,0))=\partial_y V_j(\Phi_{\mathbf{x}}^{-1}(\zeta,0))=0$. Finally, we compute
\begin{equation}
D\Phi_{\mathbf{x}}(x,y)=\begin{pmatrix}
1 & 0\\
u'_j(x) & -u_j'(y)
\end{pmatrix}\:,
\end{equation}
and so by the inverse function theorem, there holds 
\begin{equation}
D\Phi^{-1}_{\mathbf{x}}(x,y)=\frac{1}{u'_j(y)}\begin{pmatrix}
u_j'(y) & 0\\
u_j'(x) & -1
\end{pmatrix}\:,
\end{equation}
so that in particular $\partial_\eta x=0$, $\partial_\eta y=-(u_j'(y))^{-1}$. Hence there holds 
\begin{equation}
\partial^2_\eta V_j(\Phi_{\mathbf{x}}^{-1}(\zeta,0))=2(u_j'(y))^2(-(u_j'(y))^{-1})^2=2\:.
\end{equation}
\end{proof}
Finally, we have all the ingredients we need to prove the following.
\begin{theorem}
Let $\{u_j(x)\}_{j}$ be a countable set of $C^{\infty}$ functions on the torus, with a finite number of maximal overlapping critical point of order $n_0$, (i.e. for all $x \in \mathbb{T}$, there exists $j \in \mathbb{N}$ and $n \leq n_0$ so that $u_j^{(n+1)}(x) \neq 0$). Let $V=\sum_{j}(u_j(x)-u_j(y))^2$, and set $H=-\Delta+\lambda^2 V$. Then, there exists a constant $c>0$ depending only on $\{u_j\}$, so that the smallest eigenvalue of $H$ on $L^2(\mathbb{T}^2)$ is bounded below by $c\lambda^{\frac{2}{n_0+2}}$, for all $\lambda$ sufficiently large.
\end{theorem}
\begin{proof}
The proof will be divided into a series of steps. Our strategy may be summarised as follows: For each point in the zero locus of $V$, Lemma~\ref{lemma:structure of zero set} yields a local polynomial approximation to the behaviour of $V$. By the IMS localisation formula from Lemma~\ref{IMS}, it then suffices to study only the spectral asymptotics of the localised problems where the approximations from Lemma~\ref{lemma:structure of zero set} are valid. But by Lemma~\ref{lemma:Schrodinger scaling 2}, we know precisely the spectral scaling associated to Schr\"odinger operators with polynomial potentials, and hence we may deduce the required spectral scaling for each localised problem, which then completes the proof. \\
 For the sake of easing readability, we will frequently use the notation $a(f,\lambda)\lesssim b(f,\lambda)$ to mean that there exists a \emph{universal} constant $c>0$ independent of $\lambda >0, f \in H^1$, so that $a(f,\lambda) \leq c b(f,\lambda)$.

\medskip

\emph{Step 1:}
By Lemma \ref{lemma:structure of zero set}, $V^{-1}\{0\}$ may be broken down into a disjoint union 
\begin{equation}
V^{-1}\{0\}=\cup_{a=1}^n \{p_a\} \cup \Gamma_2 \cup \Gamma_3
\end{equation}
where the $\{p_a\}$ are the isolated zeros from \ref{C1}, and $\Gamma_2, \Gamma_3$ correspond to points \ref{C2}, \ref{C3} from Lemma~\ref{lemma:structure of zero set} respectively.

For any point $(x_p,y_p) \in \Gamma_2$, there exists an open ball $B_{p} \ni (x_p,y_p)$ and a diffeomorphism $\Phi_p : B_{p} \to \mathcal{U}_{p}$, for some open $\mathcal{U}_p$ with coordinates $(\zeta,\eta)$, so that $\Phi_p(V^{-1}\{0\} \cap B_{p})=\{\eta=0\}\cap \mathcal{U}_p$. Similarly, for any $q \in \Gamma_3$, we obtain a diffeomorphism $\Psi_q:B_q \to \mathcal{U}_q$ so that $V(\Psi_q^{-1}(x,y)) \geq ((x-x_q)^{n_1} \pm (y-y_q)^{n_2})^2$, for $n_1,n_2 \leq n_0+1$. Finally, for $p_a, a=1, \dots n$, there exists $B_a$ so that $V(x,y) \geq c_0((x-x_a)^{2n_0+2}+(y-y_a)^{2n_0+2})$ on $B_a$. Therefore, it holds that
\begin{equation}
V^{-1}\{0\} \subset \bigcup_{a=1}^n B_a \cup \bigcup_{p \in \Gamma_2}B_p \cup \bigcup_{q \in \Gamma_3} B_q\:.
\end{equation}
Since $V^{-1}\{0\}$ is compact, we may find a finite subcover $\{B_j\}_{j=1}^n \cup\{B_{p_i}\}_{i=1}^m \cup\{B_{q_\ell}\}_{\ell=1}^k$, and an open set $\mathcal{U}$ so that
\begin{equation}
V^{-1}\{0\} \subset \mathcal{U} \Subset  \bigcup_{a=1}^n B_a \cup \bigcup_{i=1}^m B_{p_i} \cup \bigcup_{\ell=1}^k B_{q_\ell}\:.
\end{equation}
This allows us to construct a smooth partition of unity $\{\omega_a^2\}_{a=1}^n \cup \{\phi_p^2\}_{p=1}^m \cup \{\psi^2_q\}_{q=1}^k$ of $\mathcal{U}$ subordinate to the cover $\{B_j,B_p,B_q\}$. Finally, setting $\phi_0^2=1-\sum_{a}w_a^2-\sum_p \phi_p^2 -\sum_q \psi_q^2$, we have obtained a smooth partition of unity of $[0,2\pi]$. In fact, since all functions in the partition of unity have compact support, they may be extended periodically to yield a smooth partition of unity of the $2$-torus.
\medskip

\emph{Step 2:}
Appealing to the IMS localisation formula, we see that 
\begin{equation}
H=\sum_{a}\omega_a H \omega_a+\sum_{p}\phi_pH\phi_p+\sum_{q}\psi_qH\psi_q-\sum_{a}|\nabla \omega_a|^2-\sum_{p}|\nabla \phi_p|^2-\sum_{q}|\nabla \psi_q|^2\:.
\end{equation}
We claim that there exists a constant $c>0$, so that, for all $a,p,q$ it holds that
\begin{equation}
\label{eq:IMS lower bound}
\omega_a H \omega_a \geq c\lambda^{\frac{2}{n_0+2}}\omega_a^2, \quad \phi_p H \phi_p \geq c\lambda^{\frac{2}{n_0+2}}\phi_p^2, \quad \psi_q H\psi_q \geq c\lambda^{\frac{2}{n_0+2}}\psi_q^2\:,
\end{equation}
in the sense of forms. Indeed, if this holds, then we observe that since $\sum_{a}\omega_a^2+\sum_{p}\phi_p^2+\sum_{q}\psi_q^2=1$, it holds that
\begin{equation}
H \geq c\lambda^{\frac{2}{n_0+2}} -\sum_{a}|\nabla \omega_a|^2-\sum_{p}|\nabla \phi_p|^2-\sum_{q}|\nabla \psi_q|^2 \geq \frac{1}{2}c\lambda^{\frac{2}{n_0+2}}\:,
\end{equation}
for all $\lambda$ sufficiently large, completing the proof. Therefore, we now aim to prove \eqref{eq:IMS lower bound}.

\medskip

\emph{Step 3:}
Let us deal with the $\phi_p$, $p \neq 0$ first. By assumption, there exists $j \in \mathbb{N}$, and a diffeomorphism $\Phi_p :B_p \to \mathcal{U}_p$, $\Phi_p(x,y)=(\zeta(x,y),\eta(x,y))$ so that $\Phi_p(B_p \cap \{V_j=0\})=\mathcal{U}_p \cap \{\eta=0\}$. In particular, we shall now define a partition of unity of $\supp(\phi_p)$. Indeed, fix some $\alpha>0$ which will be chosen later. Then, let $\rho :\mathbb{T} \to [0,1]$ be a smooth function so that $\rho^\lambda \equiv 1 $ on $[-\lambda^{-\alpha},\lambda^{-\alpha}]$, and $\rho^\lambda \equiv 0$ outside $(-2\lambda^{-\alpha},2\lambda^{-\alpha})$. In particular, note that $\nabla \rho^\lambda  \sim \lambda^{\alpha}$ as $\lambda \to \infty$. We then set $v_1(x,y)=\rho(\eta(x,y))$, where we recall that $\eta(x,y)$ is the second component of the diffeomorphism $\Phi_p$. Finally, set $v_0^2=1-v_1^2$. Note now that in the sense of forms it holds that $\phi_p H \phi_p \geq \phi_p H_j \phi_p$, where $H_j=-\Delta+\lambda^2 V_j(x,y)$. Further, appealing to the IMS localisation formula, it holds that
\begin{equation}
\label{eq:IMS expansion 1}
\phi_p H_j \phi_p =\phi_p v_1 (H_j-H_{j,1})v_1 \phi_p +\phi_p v_1 H_{j,1}v_1 \phi_p+\phi_p v_0 H_j v_0 \phi_p-\phi_p^2(|\nabla v_0|^2+|\nabla v_1|^2)\:,
\end{equation}
where $H_{j,1}$ is the Schrodinger operator defined by Taylor expanding $V_j$ around $\eta=0$ in the variables $(\zeta,\eta)$, i.e.
\begin{equation}
H_{j,1}=-\Delta+\lambda^2 \Tilde{V}(x,y)\:,
\end{equation}
where $\Tilde{V}(x,y)=\frac{1}{2}\eta^2 \partial^2_\eta V_j(\Phi_p^{-1}(\zeta,0))$. Fix now any function $f(x,y) \in H^1(\mathbb{T}^2)$. Then, it holds that 
\begin{equation}
\langle f,\phi_p v_1 H_{j,1}v_1 \phi_pf \rangle_{L^2} =\int_{\mathbb{T}^2}|\nabla (\phi_p v_1f)|^2+\lambda^2 |\phi_pv_1 f|^2 \Tilde{V}(x,y)\dd x \dd y\:.
\end{equation}
Let us focus for now on the term 
\begin{equation}
\int_{\mathbb{T}^2} |\phi_pv_1 f|^2 \Tilde{V}(x,y)\dd x \dd y\:.
\end{equation}
Under the change of variables $(x,y)=\Phi_p^{-1}(\zeta,\eta)$, this is equal to 
\begin{align}
&\frac{1}{2}\int_{\mathcal{U}_p}|\det(D\Phi_p^{-1})||(\phi_p v_1 f)(\Phi_p^{-1}(\zeta,\eta)|^2\eta^2 \partial_\eta^2 V(\Phi_p^{-1}(\zeta,0)) \dd \eta \dd \zeta \\
&\gtrsim \int_{\mathcal{U}_p}|(\phi_p v_1 f)(\Phi_p^{-1}(\zeta,\eta)|^2\eta^2 \dd \eta \dd \zeta\:,
\end{align}
where the inequality follows since both $\inf_{\mathcal{U}_p}|\det(D\Phi_p^{-1})|>0$ and $\inf_{\mathcal{U}_p}\partial_\eta^2 V(\Phi_p^{-1}(\zeta,0))>0$. Finally, since $\supp(\chi_p) \subset \mathcal{U}_p$, this lower bound may be replaced by 
\begin{equation}
 \int_{\mathbb{R}^2}|(\phi_p v_1 f)(\Phi_p^{-1}(\zeta,\eta)|^2\eta^2  \dd \eta \dd \zeta\:.
\end{equation}
Next, consider the term 
\begin{equation}
\int_{\mathbb{T}^2}|\nabla (\phi_p v_1 f)(x,y)|^2 \dd x \dd y\:.
\end{equation}
We once again change variables $(x,y) =\Phi_p^{-1}(\zeta,\eta)$, and bounding the Jacobian determinant below as before, we observe a lower bound of 
\begin{equation}
\int_{\mathbb{T}^2}|\nabla (\phi_p v_1 f)(x,y)|^2 \dd x \dd y \gtrsim \int_{\mathbb{T}^2}|\nabla_{x,y} (\phi_p v_1 f)(\Phi_p^{-1}(\zeta,\eta))|^2 \dd \zeta \dd \eta\:.
\end{equation}
However, a computation yields that for any smooth function $g$, there holds that
\begin{equation}
|\nabla_{x,y} g(\Phi_p^{-1}(\zeta,\eta))|^2=\nabla_{\eta,\zeta}[g(\Phi^{-1}_p(\zeta,\eta)]^TA^T A \nabla_{\eta,\zeta}[g(\Phi^{-1}_p(\zeta,\eta)]\:,
\end{equation}
where $A$ is a matrix given by 
\begin{equation}
A =\frac{1}{\displaystyle \partial_\eta x \partial_\zeta y-\partial_\zeta x \partial_\eta y} \begin{pmatrix}
\displaystyle \partial_\zeta y & \displaystyle-\partial_\eta y \\
\displaystyle -\partial_\zeta x& \partial_\eta x
\end{pmatrix}\:,
\end{equation}
where we have denoted $(x(\zeta,\eta),y(\zeta,\eta))=\Phi_p^{-1}(\zeta,\eta)$. In particular, it is clear that $A^T A$ is a positive semi-definite matrix, and further its determinant is given by 
\begin{equation}
\det(A^T A)=(\partial_\eta x\partial_\zeta y-\partial_\eta y\partial_\zeta x)^{-2}>0\:.
\end{equation}
Hence, $A^TA$ is in fact strictly positive for any $(\zeta,\eta) \in \mathcal{U}_{p}$. Furthermore, since its components are smooth in $(\zeta,\eta)$, by standard spectral perturbation theory (see e.g. \cite{K76}), the eigenvalues of $A^T A$ are continuous in $(\zeta,\eta)$, thus attaining a strictly positive minimum on $\cl(\supp(\phi_p))$. Hence, overall we see that there exists a universal constant $c>0$ so that 
\begin{equation}
|\nabla_{x,y} g(\Phi_p^{-1}(\zeta,\eta))|^2 \geq c|\nabla_{\eta,\zeta}[g(\Phi^{-1}_p(\zeta,\eta)]|^2\:,
\end{equation}
for all $g$ with support contained in $\cl(\supp(\phi_p))$. Therefore, we observe the lower bound
\begin{equation}
\int_{\mathbb{T}^2}|\nabla (\phi_p v_1 f)(x,y)|^2 \dd x \dd y \gtrsim \int_{\mathbb{R}^2}|\nabla [(\phi_p v_1 f)(\Phi_p^{-1}(\zeta,\eta))]|^2 \dd \zeta \dd \eta\:.
\end{equation}
Thus, we have shown that for any $f \in H^1(\mathbb{T}^2)$, there holds that 
\begin{equation}
\langle f,\phi_p v_1 H_{j,1}v_1 \phi_pf \rangle_{L^2(\mathbb{T}^2)}  \gtrsim  \langle \phi_pv_1 f(\Phi^{-1}_p),H_{j,2} \phi_pv_1 f(\Phi^{-1}_p)\rangle_{L^2(\mathbb{R}^2)}\:,
\end{equation}
where $H_{j,2}$ is the Schr\"odinger operator on $L^2(\mathbb{R}^2)$ defined by 
\begin{equation}
H_{j,2}=-\Delta_{\zeta,\eta}+\lambda^2  \eta^2\:.
\end{equation}
The spectrum of $H_{j,2}$ is strictly positive, and scales exactly like $\lambda$ as $\lambda \to \infty$. Therefore, it holds that 
\begin{equation}
\langle \phi_pv_1 f(\Phi^{-1}_p),H_{j,2}\phi_pv_1 f(\Phi^{-1}_p)\rangle_{L^2(\mathbb{R}^2)} \gtrsim \lambda \int_{\mathbb{R}^2}|\phi_p vf(\Phi_p^{-1})|^2 \dd \eta \dd \zeta\:.
\end{equation}
Changing variables back and bounding the resulting Jacobian below on $\supp(\phi_p)$ as before, we finally obtain that there exists a constant $c_p>0$ so that 
\begin{equation}
\phi_p v_1 H_j v_1 \phi_p \geq c_p \lambda \phi_p^2 v_1^2\:,
\end{equation}
as $\lambda \to \infty$. Next, we shall show that, up to making a good choice of $\alpha$, in some sense the remaining terms in \eqref{eq:IMS expansion 1} are small. We first deal with 
\begin{equation}
\phi_p v_1 (H_j-H_{j,1})v_1\phi_p\:.
\end{equation}
This is nothing but a multiplication operator. Furthermore, on $\supp(v_1)$, there holds that $|\eta|\leq 2 \lambda^{-\alpha}$. Therefore, recalling that $\Tilde{V}$ is simply the second order Taylor expansion of $V_j$, by Taylor's remainder formula there holds that $|H_j-H_{j,1}| \lesssim \lambda^{-3\alpha}$. For the term $\phi_p v_0 H_j v_0 \phi_p$, we once again observe via a Taylor expansion that $\lambda^{-2\alpha} \lesssim V$. Hence, there holds that
\begin{equation}
\lambda^{2-2\alpha}(\phi_p v_0)^2 \lesssim \phi_p v_0 H_j v_0 \phi_p\:.
\end{equation}
Finally, as remarked previously, $|\nabla v_0|^2+|\nabla v_1|^2 \lesssim \lambda^{2\alpha}$. Therefore, we finally see from \eqref{eq:IMS expansion 1} that 
\begin{equation}
\phi_p^2(\lambda v_1^2+\lambda^{2-2\alpha}v_0^2-\lambda^{2-3\alpha}-\lambda^{2\alpha}) \lesssim \phi_p H_j \phi_p\:.
\end{equation}
Therefore, as long as $2-2\alpha \geq 1$, $2-3\alpha<1$, $2\alpha <1$, we see that 
\begin{equation}
\phi_p^2 \lambda^{\frac{2}{n_0+2}} \leq \phi_p^2 \lambda \lesssim \phi_p H_j\phi_p \leq \phi_p H \phi_p\:.
\end{equation}
Picking for instance $\alpha=\frac{2}{5}$, these inequalities are satisfied, and hence the claim follows.

\medskip

\emph{Step 4:} We now shall deal with the terms $\psi_q H\psi_q$. In view of Lemmas \ref{lemma:schrodingerScaling1}, \ref{lemma:Schrodinger scaling 2}, \ref{lemma:LocallyPolynomial}, this case will be far easier. Indeed, we recall that by Lemma \ref{lemma:structure of zero set}, there exists a diffeomorphism $\Psi_q: B_{q}\to \mathcal{U}_q$ so that $V(\Psi_q^{-1}(x,y)) \geq ((x-x_q)^{n_1}\pm(y-y_q)^{n_2})^2$, for some $n_1,n_2 \leq n_0+1$. By precisely the same computations as in the previous step, there holds that for any $f \in H^1(\mathbb{T}^2)$
\begin{equation*}
\langle f, \psi_q H\psi_q f\rangle_{L^2(\mathbb{T}^2)} \gtrsim \int_{\mathbb{R}^2}|\nabla [\psi_qf(\Psi_q^{-1})]|^2+\lambda^2 ((x-x_q)^{n_1}\pm (y-y_q)^{n_2})^2|\psi_q f(\Psi_q^{-1})|^2 \dd x \dd y\:.
\end{equation*}
But then, by Lemma \ref{lemma:Schrodinger scaling 2},  since $n_1,n_2 \leq n_0+1$ it holds that 
\begin{equation*}
\int_{\mathbb{R}^2}|\nabla [\psi_qf(\Psi_q^{-1})]|^2+\lambda^2 ((x-x_q)^{n_1}\pm (y-y_q)^{n_2})^2|\psi_q f(\Psi_q^{-1})|^2 \dd x \dd y  \gtrsim \lambda^{\frac{2}{n_0+2}}\int_{\mathbb{R}^2}|\psi_q f(\Psi_q^{-1})|^2 \dd x \dd y\:,
\end{equation*}
which, upon changing back to our original variables, is precisely the lower bound we want. 

\medskip 

\emph{Step 5:} Finally, we discuss the terms corresponding to isolated zeros of the potential, as well as the term corresponding to $\phi_0$. Indeed, consider the term $\omega_a H \omega_a$. Since $V \geq ((x-x_a)^{2n_0+2}+(y-y_a)^{2n_0+2})$ on $\supp(\omega_a)$, we bound 
\begin{equation}
\omega_a H \omega_a \geq \omega_a (-\Delta+\lambda^2 ((x-x_a)^{2n_0+2}+(y-y_a)^{2n_0+2}))\omega_a\:.
\end{equation}
By Lemma~\ref{lemma:Schrodinger scaling 2}, we further know that $-\Delta+\lambda^2 ((x-x_a)^{2n_0+2}+(y-y_a)^{2n_0+2}) \gtrsim \lambda^{\frac{2}{n_0+2}}$ in the sense of forms, and hence we obtain the desired bound for the $\omega_a$. Finally, note that $\supp(\phi_0)$ is strictly separated from $\{V=0\}$, and therefore $V \geq c_0$ on $\supp(\phi_0)$. Therefore, it holds that
\begin{equation}
\phi_0 H \phi_0 \geq \phi_0 (-\Delta+ c_0\lambda^2)\phi_0 \gtrsim  \phi_0^2 \lambda^2\:,
\end{equation}
and hence we have completed the proof. 
\end{proof}

\subsection{A Lower Bound on the Dissipation Rate}
Obtaining a lower bound on the enhanced dissipation rate will once again require studying the spectral asymptotics of certain classes of Schrödinger operators, but this time the situation is much nicer, as we only need to deal with the operator generated by the one-point motion.
\begin{lemma}
Assume that the $\{u_j\}_{j \in \mathbb{N}}$ have a spatially overlapping critical point of order $n_0$ at $y=0$, and assume further that $\sum_{j}\|u_j\|_{C^{2n_0+2}}<\infty$. Then, there exists a family of initial data \(f^{\alpha}_0\) depending on \(\alpha =\frac{|\ell|\sqrt{\kappa}}{\sqrt{\nu}}\), so that for \(\alpha\) large enough, the solution to \eqref{Eq:StochasticShear}
satisfies
\begin{equation}\label{eq:lowerbdden}
\EE\|P_\ell f^\nu(t)\|_{L^2} \geq \|P_\ell f_0^\alpha\|_{L^2} \exp{(-c\nu^{\frac{n_0+1}{n_0+2}}\kappa^\frac{1}{n_0+2}|\ell|^\frac{2}{n_0+2}t)}\:,
\end{equation}
where \(c\) is a constant independent of \(\alpha\).
\end{lemma}
\begin{proof}
Let $f(x,y,t)$ be the solution to the SPDE \eqref{Eq:StochasticShear} with initial condition $f(x,y)$. Set now
$$
v(t,x,y)=f\Bigl(x+\sqrt{2\kappa}\sum_{j\in \mathbb{N}}u_j(0)W^k_t,y,t\Bigr)\:. 
$$
Note that $|P_{\ell}v(t,y)|=|P_{\ell}f(t,y)|$, for all $\ell$. Hence, it suffices to prove the bound \eqref{eq:lowerbdden} for $v$. Now, a direct application of the Itô formula shows that \(v\) satisfies the equation 
\begin{equation}
\dd v=-\sqrt{2\kappa}\sum_{j \in \mathbb{N}} (u_j(y)-u_j(0))\partial_x v\circ \dd W^j_t +\nu \partial^2_y v\, \dd t\:.
\end{equation}
Computing the Stratonovich correction, taking the Fourier transform in \(x\) and then finally taking expectations, we see that \(g_\ell(t,y)=\EE(\hat{v}_\ell(t,y))\) satisfies
\begin{equation}
\label{eq:fourier shear 2}
\partial_t g_\ell =-\ell^2\kappa \sum_{j \in \mathbb{N}}(u_j(y)-u_j(0))^2g_\ell+\nu \partial^2_y g_\ell\:.
\end{equation}
We thus see that \(g_\ell\) satisfies a Schrödinger type equation, with non-negative potential $V(y)=\sum_{j \in \mathbb{N}}(u_j(y)-u_j(0))^2$ that vanishes of order \(2n_0+2\) at \(y=0\). Thus, it is a matter of showing that there exist eigenfunctions of this operator with associated eigenvalues that scale like $\nu^{\frac{n_0+1}{n_0+2}}\kappa^{\frac{1}{n_0+2}}|\ell|^\frac{2}{n_0+2}$.
We consider the operators
\begin{equation}
H=-\partial^2_y+\frac{V^{(2n_0+2)}(0)}{(2n_0+2)!}\lambda^2y^{2n_0+2}
\end{equation}
on the real line with $\lambda^2=\ell^2\kappa\nu^{-1}$. Letting \(p=p(y)\) satisfy \((-\partial^2_y+\frac{V^{(2n_0+2)}(0)}{(2n_0+2)!}y^{2n_0+2})p=\mu p\), normalized to $\|p\|_{L^2(\mathbb{R})}=1$, it is clear that 
\begin{equation}
p^\lambda(y)=p(\lambda^\frac{1}{n_0+2}y)
\end{equation}
satisfies
\begin{equation}
Hp^\lambda=\mu\lambda^\frac{2}{n_0+2}p^\lambda\:.
\end{equation}
Now, let \(\chi(y)\) be a smooth function which equals \(1\) on \([-\frac{4\pi}{5},\frac{4\pi}{5}]\), and vanishes outside \([-\frac{8\pi}{9},\frac{8\pi}{9}]\). Then, for $\beta <\frac{1}{n_0+2}$, consider the function 
\begin{equation}
q^\lambda(y)=\frac{1}{\|\chi(\lambda^{\beta}\cdot)p^\lambda \|_{L^2(\mathbb{T})}}\chi(\lambda^{\beta} y)p^\lambda(y)\:,
\end{equation}
which has compact support on \([-\pi,\pi]\), and thus can be extended to a periodic \(H^1\) function. Then, denoting \(V=\sum_{j \in \mathbb{N}}(u_j(y)-u_j(0))^2\), consider
\begin{equation}
\int_{-\pi}^\pi |\partial_y q^\lambda(y)|^2\dd y+\lambda^2\int_{-\pi}^\pi V(y)(q^\lambda(y))^2\dd y\:.
\end{equation}
We shall see how both of these terms scale in terms of $\lambda$. 
Note first that
\begin{equation}
\|\chi(\lambda^\beta y)p^\lambda(y)\|^2_{L^2(\mathbb{T})} \geq  \int_{-\lambda^{-\beta}\frac{4\pi}{5}}^{\lambda^{-\beta}\frac{4\pi}{5}} |p(\lambda^\frac{1}{n_0+2}y)|^2\dd y =\lambda^{-\frac{1}{n_0+2}}\int_{-\lambda^{\frac{1}{n_0+2}-\beta}\frac{4\pi}{5}}^{\lambda^{\frac{1}{n_0+2}-\beta}\frac{4\pi}{5}} |p(y)|^2\dd y \geq c\lambda^{-\frac{1}{n_0+2}}\:.
\end{equation}
Furthermore, we have
\begin{align}
\int_{-\pi}^{\pi}|\partial_y(\chi(\lambda^{\beta}y)p(\lambda^\frac{1}{n_0+2}y))|^2 \dd y 
&\lesssim \int_{-\pi}^\pi \lambda^{2\beta} \|\partial_y \chi\|_{L^\infty}^2 |p(\lambda^\frac{1}{n_0+2}y)|^2\dd y +\int_{-\pi}^\pi |\partial_y p(\lambda^\frac{1}{n_0+2}y)|^2\lambda^\frac{2}{n_0+2}\dd y \notag\\
&\lesssim \lambda^{2\beta-\frac{1}{n_0+2}}+\lambda^{\frac{1}{n_0+2}}\|\partial_y p\|_{L^2(\mathbb{R})}^2\:.
\end{align}
Hence, overall it holds that
\begin{equation}
\int_{-\pi}^\pi |\partial_y q^\lambda(y)|^2\dd y \lesssim \lambda^\frac{2}{n_0+2}\|p\|_{H^1(\mathbb{R})}^2\:,
\end{equation}
as long as we impose the constraint
\begin{equation}
\beta < \frac{1}{n_0+2}\:.
\end{equation}
Next, consider
\begin{align}
\label{eq:potential lower bound}
\int_{-\pi}^{\pi}V(y)\left|\chi(\lambda^{\beta}y) p(\lambda^\frac{1}{n_0+2}y)\right|^2\dd y 
&\leq \int_{-\lambda^{-\beta}\pi}^{\lambda^{-\beta} \pi}V(y)|p(\lambda^\frac{1}{n_0+2})|^2\dd y\notag\\
&\leq \lambda^{-\frac{1}{n_0+2}}\int_{-\pi\lambda^{\frac{1}{n_0+2}-\beta}}^{\pi \lambda^{\frac{1}{n_0+2}-\beta}}V(\lambda^{-\frac{1}{n_0+2}}y)|p(y)|^2\dd y\:.
\end{align}
Now, it holds that \(|V(y)-\frac{V^{(2n_0+2)}(0)}{(2n_0+2)!}y^{2n_0+2}| \lesssim y^{2n_0+3}\), and so, for \(y \in [-\lambda^{-\beta}\pi,\lambda^{-\beta}\pi]\), we have
\begin{equation}
\left\|V(y)-\frac{V^{(2n_0+2)}(0)}{(2n_0+2)!}y^{2n_0+2}\right\|_{L^\infty} \lesssim \lambda^{-(2n_0+3)\beta}\:.
\end{equation}
Hence, we may upper bound the expression in \eqref{eq:potential lower bound} by
\begin{equation}
\int_{\mathbb{R}}\lambda^{-(2n_0+3)\beta}|p(y)|^2\dd y+\int_{\mathbb{R}}\frac{V^{(2n_0+2)}(0)}{(2n_0+2)!}\lambda^{-\frac{2n_0+2}{n_0+2}}y^{2n_0+2}|p(y)|^2\dd y \lesssim \lambda^{-\frac{2n_0+2}{n_0+2}}\:,
\end{equation}
under the additional constraint
\begin{equation}
\beta > \frac{2n_0+2}{(n_0+2)(2n_0+3)}\:.
\end{equation}
Note that the two constraints for $\beta$ we imposed may be satisfied simultaneously. Indeed, it always holds that $ \frac{2n_0+2}{(n_0+2)(2n_0+3)} < \frac{1}{n_0+2}$.
Hence, overall we deduce  
\begin{equation}
\lambda^2\int_{\pi}^{\pi}V(y)q^\lambda(y)^2\dd y \lesssim \lambda^{2-\frac{2n_0+2}{n_0+2}} =\lambda^{\frac{2}{n_0+2}}\:.
\end{equation}
Thus, we have shown that
\begin{equation}
\inf_{q \in H^1(\mathbb{T}), \|q\|_{L^2}=1}\left[\int_{\mathbb{T}}|\partial_yq(y)|^2\dd y+\lambda^2 \int_{\mathbb{T}}V(y)|q(y)|^2\dd y\right]\leq C \lambda^{\frac{2}{n_0+2}}\:,
\end{equation}
for some \(C\) independent of \(\lambda\). Finally, by compactness of the resolvent we know that the eigenvalues are attained. Thus, there exists some sequence \(g^\lambda\) of \(H^1\) functions satisfying 
\begin{equation}
H g^\lambda = \mu(\lambda) g^\lambda\:.
\end{equation}
where $\mu(\lambda) \leq c\lambda^{\frac{2}{n_0+2}}$.
The solution to the equation \eqref{eq:fourier shear 2} with initial datum \(r^\lambda\) is then simply 
\begin{equation}
g_\ell(t,y)=g^\lambda(y)e^{-\nu \mu(\lambda)t}\:.
\end{equation}
Taking the $L^2$ norm, we thus recover
\begin{equation}
\|g_\ell\|_{L^2}\geq \|g^\lambda(y)\|_{L^2}\e^{-c\nu^\frac{n_0+1}{n_0+2}\kappa^\frac{1}{n_0+2}|\ell|^\frac{2}{n_0+2}t}\:.
\end{equation}
Using the convexity of the \(L^2\) norm, an application of Jensen's inequality yields the desired result. 
\end{proof}

\appendix
\section{Some Technical Results}\label{app:techno}
We collect here a series of technical results which have been referenced throughout the paper. In particular, Section~\ref{Justification} will deal with formalising the arguments in the proof of our Stochastic RAGE theorem, and providing some of the necessary functional analytic details.

\subsection{Justification of the Formal Computation}\label{Justification}
Here we show that under the assumptions we placed on the inviscid SPDE, the formal computations which were carried out in the introduction may be justified. We begin by recalling the classical identification of tensor product of Hilbert spaces, and the space of Hilbert--Schmidt operators on them. In particular, we recall that the real space \(\sym\) is isometrically isomorphic to the space of self-adjoint Hilbert--Schmidt operators on \(H\).
\begin{lemma}
\label{Hilbert--Schmidt isomorphism}
The space \(\sym\) is isometrically isomorphic to the space of self-adjoint Hilbert--Schmidt operators on \(H\) under the map \(T(x\otimes y) (z) =\langle z,\overline{y}\rangle x\).
\end{lemma}
Next, we provide a proof of the following Itô formula for tensor products of stochastic processes.
\begin{lemma}
\label{Ito formula}
Let $H$ be a separable Hilbert space, and let $W$ be a $Q$-Wiener processes on $H$, where $Q$ is some positive trace class operator on $H$. Let $\Psi^i_t$ $i=1,2$ be some $HS(Q^\frac{1}{2}(H),H)$ valued processes that are stochastically integrable with respect to $W$, and suppose that $x_i$, $\phi_i$ are some $H$-valued predictable processes that are further Bochner integrable, and satisfy
\begin{equation}
x_i(t)=x_0+\int_0^t \phi_i(s)ds+\int_0^t \Psi_i(s)\dd W_s\:.
\end{equation}
Note that in particular, $\Phi_2 Q \Phi_1^*$ is Hilbert--Schmidt, and therefore it may be viewed as an element in $H\otimes H$. Then, under this identification, the tensor product $x_1\otimes x_2$ satisfies 
\begin{align}
x_1(t)\otimes x_2(t)&=x_1(0)\otimes x_2(0)+\int_0^t x_1(s)\otimes \psi_2(s)+\psi_1(s)\otimes x_2(s)\dd s\notag\\
&\quad +\int_0^t \Phi_1(s)\otimes x_2(s) \dd W_s +\int_0^t x_1(s)\otimes \Phi_2(s) \dd W_s +\int_0^t \Phi_1(s)Q\Phi_2^*(s)\dd s\:,
\end{align}
where we have denoted by $\Phi_1\otimes x_2$ the Hilbert--Schmidt operator $Q^\frac{1}{2}(H)\to H\otimes H$ acting via $\Phi_1\otimes x_2(a)=\Phi_1(a)\otimes x_1$, and similarly for $x_1\otimes \Phi_2$. Note in particular that stochastic integration against these processes is still well defined in the sense of \cite{Prato_Zabczyk_2014}.
\end{lemma}
\begin{proof}
This follows immediately upon applying the usual infinite dimensional Ito formula to the map $ x 
\oplus x \to x \otimes x$.
\end{proof}
With this formula established, we may apply it to strong solutions of our SPDE to obtain the following.
\begin{lemma}
Let $f_0 \in \mathcal{C}$, and let $f(t)$ be the strong solution to the SPDE \eqref{eq:InviscidAbstract} with initial condition $f_0$. Then, letting $\Tilde{L}$ be any negative, self-adjoint extension of $L=\sum_k (iL_k\otimes \Id+\Id\otimes (iL_k))^2$, there holds that
\begin{equation}
\mathbb{E}(f(t)\otimes \overline{f(t)})=\e^{\Tilde{L}t}\mathbb{E}(f_0\otimes \overline{f_0})\:.
\end{equation}
\end{lemma}
\begin{proof}
By our definition of a strong solution, we may apply the Itô formula for tensor products from Lemma~\ref{Ito formula} to get
\begin{equation}
\dd (f\otimes \overline{f})=\dd M_t+\frac{1}{2}\sum_{k}(iL_k\otimes \Id -\Id\otimes i\overline{L_k})^2f\otimes \overline{f}\dd t,
\end{equation}
where \(M_t\) is a martingale, by our definition of a strong solution.
We now construct a self-adjoint extension \(\widetilde{L}\) on \(\sym\) of the operator \(\frac{1}{2}\sum_{k}(iL_k\otimes \Id -\Id\otimes i\overline{L_k})^2\), using for instance the Friedrichs extension, see \cite{Reed_Simon_1972}. Since all we are doing is taking extensions, it still holds that 
\begin{equation}
\dd (f\otimes \overline{f})=\dd M_t+\widetilde{L}(f\otimes \overline{f})\dd t\:.
\end{equation}
Finally, we want to interchange the expectation and the action of the closed operator \(\widetilde{L}\). By Hille's Theorem for Bochner integrals, it remains to check that \(\EE\int_0^T\|\widetilde{L}f(t)\otimes \overline{f(t)}\|\dd t<\infty\). A calculation shows that the norm in the above expression is equal to
\begin{equation}
\label{eq:norm bound hille}
\biggl(\sum_{k,j}\langle L_k^2f,L_j^2f\rangle \|f\|^2+\|L_kf\|^2\|L_jf\|^2+2|\langle L_jf,L_kf\rangle|^2+2\Re\left(\langle L_k^2f,L_jf\rangle \langle f, L_j f \rangle\right) \biggr)^\frac{1}{2}\:.
\end{equation}
The last term may be bounded above by
\begin{equation}
\sum_{j}\Bigl\|\sum_{k}L_k^2f\Bigr\|\|L_jf\|^2\|f\| \leq \frac{1}{2}\Bigl(\sum_{j}\|L_jf\|^2\Bigr)^2 +\frac{1}{2}\|f\|^2\Bigl(\|\sum_{k}L_k^2f\|\Bigl)^2\:,
\end{equation}
and so, since \(\sum_{k,j}\langle L_k^2f,L_j^2f\rangle =\|\sum_{k}L_k^2f\|^2\), we see that the expression \eqref{eq:norm bound hille}
can be bounded by
\begin{equation}
2\|f\| \Bigl(\|\sum_{k}L_k^2f\|^2\Bigr)^\frac{1}{2}+2\sum_{k}\|L_kf\|^2+2\Bigl(\sum_{k,j}|\langle L_jf,L_kf\rangle|^2\Bigr)^\frac{1}{2}\:.
\end{equation}
Taking expectation and applying Cauchy--Schwarz, we see that 
\begin{equation}
\EE\int_0^T\|\widetilde{L}f(t)\otimes \overline{f(t)}\|\dd t\leq 2 \EE\|f(t)\|^2\int_0^T \|\sum_{k}L_k^2f(t)\|^2 \dd t+4\sum_{k}\int_{0}^T\EE\|L_kf(t)\|^2 \dd t\:.
\end{equation}
By our definition of a strong solution, this is indeed finite for any $T >0$, and so we may apply Hille's Theorem as desired. Finally, observe that $\mathbb{E}(f(t)\otimes \overline{f(t)})=e^{\Tilde{L}t}\mathbb{E}(f(0)\otimes \overline{f(0)})$. This follows by noting that $\mathbb{E}(f(t) \otimes \overline{f(t)}) \in \mathcal{D}(\tilde{L})$ for all $t \geq 0$, and so it holds that
\begin{equation}
\label{eq:semigroup expectation}
\mathbb{E}(f(t)\otimes \overline{f(t)})=\mathbb{E}(f_0\otimes \overline{f(t)})+\int_0^t \tilde{L}\mathbb{E}(f(s)\otimes \overline{f(s)})\dd s\:.
\end{equation}
Similarly, since $\mathbb{E}(f_0 \otimes \overline{f_0}) \in \mathcal{D}(\tilde{L})$, the equality \eqref{eq:semigroup expectation} also holds for $\e^{\Tilde{L}t}\mathbb{E}(f_0 \otimes \overline{f_0})$. Therefore, setting $d(t)=\mathbb{E}(f(t)\otimes \overline{f(t)})-e^{\Tilde{L}t}\mathbb{E}(f_0 \otimes \overline{f_0})$, it satisfies 
\begin{equation}
d(t)=\int_0^t Ld(s) \dd s\:,
\end{equation}
and so 
\begin{equation}
\partial_t\|d(t)\|^2=\langle L d(t),d(t)\rangle \leq 0\:,
\end{equation}
since $d(t) \in \mathcal{D}(\tilde{L})$ for all $t \geq 0$.
Hence $d(t) \equiv 0$, and so the proof is over.
\end{proof}
Hence we have shown that our formal computation indeed holds for strong solutions to the SPDE. Thus, if we can show that strong solutions are ``dense", this will prove that it extends to any initial condition. This is the subject of the following lemma.
\begin{lemma}
\label{semigroups are unique}
Assume that there exists a norm dense subset \(\cC\subset L^2(\Sigma, \cF,\PP,H)\), so that any solution to the SPDE \eqref{eq:InviscidAbstract} with data in \(\cC\) is a strong solution. Then, for any closed extension \(\hat{L}\) of \(L\), and any initial condition $f$, it holds that \(\EE(f(t)\otimes \overline{f}(t))=\e^{
\hat{L}t}\EE(f\otimes \overline{f})\), where \(\e^{\hat{L}t}\) is the semigroup generated by \(\hat{L}\). In particular, since finite spans of elements of the form $f \otimes \overline{f}$ are dense in $\sym$, all self-adjoint, positive extensions of \(L\) on \(\sym\) must be the same, and hence \(L\) is essentially self-adjoint.
\end{lemma}
\begin{proof}
We show that if $f_n \to f \in L^2(\Omega,\mathcal{F},\mathbb{P},H)$, then \(\EE(f_n(t)\otimes \overline{f}_n(t))\to \EE(f(t)\otimes \overline{f}(t)) \). Indeed,  
\begin{align}
\EE\|f_n(t)&\otimes \overline{f}_n(t)-f(t)\otimes \overline{f}(t)\| \\
&\leq \EE\|f_n(t)\otimes \overline{f}_n(t)-f(t)\otimes \overline{f}_n(t)\|+\EE\|f(t)\otimes \overline{f}_n(t)-f(t)\otimes \overline{f}(t)\|\notag\\
&\leq 2\EE(\|f_n(t)-f(t)\|\|f_n(t)\|)  \to 0,
\end{align}
as \(n \to \infty\). Hence, since \(\EE(f_n(t)\otimes \overline{f}_n(t))=\e^{\hat{L}t}\EE(f_n\otimes \overline{f_n})\), taking \(n \to \infty\) and using continuity, we see that 
\begin{equation}
\EE(f(t)\otimes \overline{f}(t))=\e^{\hat{L}t}\EE(f\otimes \overline{f})\:,
\end{equation}
concluding the proof.
\end{proof}
\begin{lemma}\label{self-adjointtensor}
Let \(A,B\) be two self-adjoint, unbounded operators on a Hilbert space \(H\) so that \(\cD(A^m)\cap \cD(B^m)\) is dense in \(H\), for some \(m\in \mathbb{N}\). Let \(D_0\) be the finite span of the set \(\{\phi\otimes \varphi: \phi \in \cD(A^m), \varphi\in \cD(B^m)\}\), and define the operator \((A\otimes \Id+\Id\otimes B)^m:D_0\to H\otimes H\) by \(\sum_{j=0}^m \binom{m}{j} A^j\phi\otimes B^{m-j}\varphi\) for \(\phi\otimes \varphi \in D_0\). Then, the closure of this operator is self-adjoint, i.e. \(D_0\) is a core for \((A\otimes \Id+\Id\otimes B)^m\). In particular, if \(A,B\) are positive and have a compact resolvent, then the same holds true for \((A\otimes \Id+\Id\otimes B)^m\).
\end{lemma}
\begin{proof}
By the spectral theorem, there exists unitary maps \(U_i :H \to L^2(\Sigma_i,\mu_i)\) under which the operators \(A,B\) are multiplication operators (say with associated functions \(f_A,f_B\)). Furthermore, we can take the \(\mu_i\) to be finite measures. Then, define \(U:H\otimes H \to L^2(\Sigma_1,\mu_1)\otimes L^2(\Sigma_2,\mu_2)\) by \(U=U_1\otimes U_2\). By the finiteness of the measures, it holds that \(L^2(\Sigma_1,\mu_1)\otimes L^2(\Sigma_2,\mu_2) \cong L^2(\Sigma_1\times \Sigma_2,\mu_1\otimes \mu_2)\), and it is easy to see that under \(U\), \((A\otimes \Id+\Id\otimes B)^m=(f_A(x)+f_B(y))^m\). From this the first claim follows. 

For the final result, we use the characterisation of compact multiplication operators given by \cite{Takagi_1992}, which says that the $\theta$-multiplication operator \(M_\theta :L^2_\mu \to L^2_\mu\) is compact if and only if the restriction of \(L^2_\mu\) to \(S^\eps=\{x:|\theta(x)|\geq \eps\}\) is finite dimensional for any \(\eps \geq 0\). Hence, since \(0 \in \rho(A) \cap \rho(B)\), where $\rho$ denotes the resolvent set, we have that for \(S^\eps_A=\{x:|f_A(x)^{-1}|\geq \eps\}\) (and similarly defining \(S_B^\eps\)), the spaces \(L^2(S_A^\eps,\mu_A)\), \(L^2(S_B^\eps,\mu_B)\) are finite dimensional. But then the same holds for \(L^2(S_A^\eps,\mu_A) \otimes L^2(S_B^\eps,\mu_B) \cong L^2(S_A^\eps \times S_B^\eps,\mu_A\otimes \mu_B)\). Finally, note that by the positivity of \(A,B\), 
$$
S^\eps=\{(x,y):\frac{1}{|f_A(x)+f_B(y)|^m} \geq \eps \} \subset S_A^{\eps^\frac{1}{m}}\times S_B^{\eps^\frac{1}{m}}\:,
$$
so \(L^2(S^\eps,\mu_A\otimes \mu_B)\subset L^2(S_A^{\eps^\frac{1}{m}} \times S_B^{\eps^\frac{1}{m}},
\mu_A\otimes \mu_B)\) is finite dimensional, completing the proof.
\end{proof}
See also \cite{Reed_Simon_1973} for a more general discussion of tensor products of unbounded operators.

\addtocontents{toc}{\protect\setcounter{tocdepth}{0}}
\section*{Acknowledgments}
The research of MCZ was partially supported by the Royal Society URF\textbackslash R1\textbackslash 191492 and ERC-EPSRC Horizon Europe Guarantee EP/X020886/1. The research of DV was funded by the Imperial College President’s PhD Scholarships.

\addtocontents{toc}{\protect\setcounter{tocdepth}{1}}

\bibliographystyle{Martin}
\bibliography{bib.bib}

\end{document}